\documentclass{article}
\usepackage[latin1]{inputenc}
\usepackage[T1]{fontenc}
\usepackage{amsmath,amssymb}
\usepackage{amsfonts,amsthm}
\usepackage{url}
\usepackage{geometry}
\usepackage[hypertexnames=false]{hyperref}
\usepackage{hypcap}
\usepackage{color}
\usepackage{graphicx}
\usepackage{caption}
\usepackage{epstopdf}

\newcommand{\C}{\mathbb{C}}
\newcommand{\I}{\mathbb{I}}
\newcommand{\N}{\mathbb{N}}

\newcommand{\R}{\mathbb{R}}
\newcommand{\T}{\mathbb{T}}
\newcommand{\Z}{\mathbb{Z}}

\newcommand{\Xb}{\textbf{\upshape X}}

\newcommand{\Kc}{\mathcal{K}}
\newcommand{\Lc}{\mathcal{L}}

\newcommand{\from}{\colon}

\DeclareMathOperator{\im}{im}

\DeclareMathOperator{\op}{op}
\DeclareMathOperator{\ess}{ess}
\DeclareMathOperator{\BO}{BO}
\DeclareMathOperator{\clos}{clos}
\DeclareMathOperator{\conv}{conv}
\DeclareMathOperator{\Real}{Re}
\DeclareMathOperator{\Imag}{Im}
\DeclareMathOperator{\spec}{sp}
\DeclareMathOperator{\spess}{\spec_{\ess}}

\DeclareMathOperator{\sign}{sign}
\DeclareMathOperator{\PsiE}{\Psi \mathrm{E}}

\renewcommand{\Re}{\Real}
\renewcommand{\Im}{\Imag}
\renewcommand{\epsilon}{\varepsilon}
\renewcommand{\phi}{\varphi}
\renewcommand{\theta}{\vartheta}

\providecommand{\scpr}[2]{\left\langle #1, #2 \right\rangle}
\renewcommand{\sp}{\scpr}
\providecommand{\abs}[1]{\left\lvert#1\right\rvert}
\providecommand{\norm}[1]{\left\lVert#1\right\rVert}
\providecommand{\set}[1]{\left\{ #1\right\}}

\newtheorem{thm}{Theorem}

\newtheorem{lem}[thm]{Lemma}
\newtheorem{prop}[thm]{Proposition}
\newtheorem{cor}[thm]{Corollary}

\theoremstyle{definition}
\newtheorem{defn}[thm]{Definition}

\newtheorem{ex}[thm]{Example}

\begin{document}

\title{On the Spectrum and Numerical Range \\ of Tridiagonal Random Operators}
\author{Raffael Hagger\footnote{raffael.hagger@tuhh.de}}
\maketitle

\begin{abstract}
In this paper we derive an explicit formula for the numerical range of (non-self-adjoint) tridiagonal random operators. As a corollary we obtain that the numerical range of such an operator is always the convex hull of its spectrum, this (surprisingly) holding whether or not the random operator is normal. Furthermore, we introduce a method to compute numerical ranges of (not necessarily random) tridiagonal operators that is based on the Schur test. In a somewhat combinatorial approach we use this method to compute the numerical range of the square of the (generalized) Feinberg-Zee random hopping matrix to obtain an improved upper bound to the spectrum. In particular, we show that the spectrum of the Feinberg-Zee random hopping matrix is not convex.
\end{abstract}



\section{Introduction}

Since the introduction of random operators to nuclear physics by Eugene Wigner \cite{Wigner} in 1955, there is an ongoing interest in random quantum systems, the most famous example probably being the Anderson model \cite{Anderson}. In the last twenty years also non-self-adjoint random systems were extensively studied, starting with the work of Hatano and Nelson \cite{HaNe}. Compared to self-adjoint random operators, non-self-adjoint random operators give rise to many new phenomena like complex spectra, (non-trivial) pseudospectra, etc. In return, the study of non-self-adjoint operators requires new techniques as the standard methods from spectral theory are often not available. 

We start with some limit operator and approximation results for numerical ranges of random operators. We then focus on the physically most relevant case of tridiagonal operators. In particular, we prove an easy formula for the (closure of the) numerical range of tridiagonal random operators (Theorem \ref{ntridiag}). As a corollary we get that the (closure of the) numerical range is equal to the convex hull of the spectrum for these operators, just like for self-adjoint or normal operators. Theorem \ref{ntridiag} thus provides the best possible convex upper bound to the spectrum of a random tridiagonal operator. In particular, it improves the upper bound given in \cite{ChaDa} for a particular class of random tridiagonal operators. The authors of \cite{ChaDa} considered the following tridiagonal random operator:
\[\begin{pmatrix} \ddots & \ddots & & & & \\ \ddots & 0 & 1 & & & \\ & c_{-1} & 0 & 1 & & \\ & & c_0 & 0 & 1 & \\ & & & c_1 & 0 & \smash{\ddots} \\ & & & & \ddots & \ddots \end{pmatrix},\]
where $(c_j)_{j \in \Z}$ is a sequence of i.i.d.~random variables taking values in $\set{\pm \sigma}$ and $\sigma \in (0,1]$. The special case $\sigma = 1$ was already considered earlier (e.g.~in \cite{ChaChoLi}, \cite{ChaChoLi2}, \cite{FeZee}, \cite{HoOrZee}) and is called the Feinberg-Zee random hopping matrix. It is also the main topic of \cite{Ha} and \cite{Ha2}, where the symmetries of the spectrum and the connections to the spectra of finite sections of this operator are studied, respectively.

Theorem \ref{ntridiag} also determines the spectrum completely in some cases. Consider for example the Hatano-Nelson operator
\[A = \begin{pmatrix} \ddots & \ddots & & & & \\ \ddots & v_{-1} & e^g & & & \\ & e^{-g} & v_0 & e^g & & \\ & & e^{-g} & v_1 & e^g & \\ & & & e^{-g} & v_2 & \smash{\ddots} \\ & & & & \ddots & \ddots \end{pmatrix},\]
where $(v_j)_{j \in \Z}$ is a sequence of i.i.d.~random variables taking values in some bounded set $V \subset \R$ and $g > 0$ is a constant, and assume that $V$ is an interval of length at least $4\cosh(g)$. Then Theorem \ref{ntridiag} implies that the spectrum of $A$ is equal to the numerical range, which is given by the union of the ellipses $E_v := \set{e^{g+i\theta} + v + e^{-(g+i\theta)} : \theta \in [0,2\pi)}$, $v \in V$.

In Section \ref{method} we introduce a method to compute numerical ranges of arbitrary (not necessarily random) tridiagonal operators that is based on the Schur test. For the (generalized) Feinberg-Zee random hopping matrix as studied in \cite{ChaDa} and mentioned above, we use this method to compute the numerical range of the square of the random operator, which will provide an improved upper bound to the spectrum. This is related to the concept of higher order numerical ranges as used in \cite{Davies0.5} and \cite{Martinez} for example. 

In the last part we provide explicit formulas for the numerical range and the numerical range of the square in the case of the (generalized) Feinberg-Zee random hopping matrix in order to show that this new upper bound is indeed a tighter bound to the spectrum than the numerical range. In particular, we confirm and improve the numerical results obtained in \cite{ChaChoLi2} concerning the question whether the spectrum is equal to the (closure of the) numerical range in the case $\sigma = 1$. More precisely, we show that the spectrum is a proper subset of the (closure of the) numerical range and not convex.

\subsection{Notation}

Throughout this paper we consider the Hilbert space $\Xb := \ell^2(\Z)$ and its closed subspace $\ell^2(\N)$. The set of all bounded linear operators $\Xb \to \Xb$ will be denoted by $\Lc(\Xb)$. The set of all compact operators $\Xb \to \Xb$ will be denoted by $\Kc(\Xb)$.

We want to think of $\Lc(\Xb)$ as a space of infinite matrices. Operators in $\Lc(\Xb)$ are identified with infinite matrices in the following way. Let $\left\langle \cdot,\cdot \right\rangle$ be a scalar product defined on $\Xb$ and let $\left\{e_i\right\}_{i \in \Z}$ be a corresponding orthonormal basis, i.e.~$\left\langle e_i,e_j \right\rangle = \delta_{i,j}$ for all $i,j \in \Z$. We will keep this orthonormal basis fixed for the rest of the paper. The subsequent notions may depend on the chosen basis.

Let $A \in \Lc(\Xb)$. Then the entry $A_{i,j}$ is given by $\left\langle Ae_j,e_i \right\rangle$. The matrix $\left(A_{i,j}\right)_{i,j \in \Z}$, in the following again denoted by $A$, acts on a vector $v \in \Xb$ in the usual way. If $v_j$ is the $j$-th component of $v$, then the $i$-th component of $Av$ is given by $\sum\limits_{j \in \Z} A_{i,j}v_j$. This identification of operators and matrices on $\Xb$ is an isomorphism (see e.g.~\cite[Section 1.3.5]{Marko}). Therefore we do not distinguish between operators and matrices. As usual, the vector $(A_{i,j})_{j \in \Z} \in \Xb$ is called the $i$-th row and $(A_{i,j})_{i \in \Z} \in \Xb$ is called the $j$-th column of $A$. For $k \in \Z$ the vector $(A_{i+k,i})_{i \in \Z} \in \Xb$ is called the $k$-th diagonal of $A$ or the diagonal with index $k$. $A$ is called a band operator if only a finite number of diagonals are non-zero. The set of all band operators will be denoted by $\BO(\Xb)$. Furthermore, we call $A$ tridiagonal if all diagonals with index $k \notin \set{-1,0,1}$ vanish.

We consider the following subclasses. Let $n \leq m$ be integers and let $U_n, \ldots, U_m \subset \C$ be non-empty compact sets. Then we define
\[M(U_n, \ldots, U_m) = \left\{A \in \Lc(\Xb) : A_{i+k,i} \in U_k \text{ if } n \leq k \leq m \text{ and } A_{i+k,i} = 0 \text{ otherwise}\right\},\]
i.e.~the $k$-th diagonal only contains elements from $U_k$. Similarly, we denote the set of all finite square matrices with this property by $M_{fin}(U_n, \ldots, U_m)$. If $A \in M(U_n, \ldots, U_m)$ satisfies $A_{i,j} = A_{i+p,j+p}$ for all $i,j \in \Z$ and some $p \geq 1$, then $A$ is called $p$-periodic and the set of all of these operators will be denoted by $M_{per,p}(U_n, \ldots, U_m)$. In the special case $p = 1$ these operators are usually called Laurent operators and therefore we additionally define $L(U_n, \ldots, U_m) := M_{per,1}(U_n, \ldots, U_m)$. The set of all periodic operators will be denoted by $M_{per}(U_n, \ldots, U_m)$.

$A \in M(U_n, \ldots, U_m)$ is called a random operator if for $k \in \left\{n, \ldots, m\right\}$ the entries along the $k$-th diagonal of $A$ are chosen randomly (say i.i.d.)~w.r.t.~some probability measure on $U_k$. Finally, pseudo-ergodic operators are defined as follows. Let $P_{k,l}$ be the orthogonal projection onto $\text{span}\left\{e_k, \ldots, e_l\right\}$. Then $A \in M(U_n, \ldots, U_m)$ is called pseudo-ergodic if for all $\epsilon > 0$ and all $B \in M_{fin}(U_n, \ldots, U_m)$ there exist $k$ and $l$ such that $\left\|P_{k,l}AP_{k,l} - B\right\| \leq \epsilon$. In other words, every finite square matrix of this particular kind can be found up to epsilon when moving along the diagonal of a pseudo-ergodic operator. Note that if all of the $U_k$ are discrete, one can simply put $\epsilon = 0$ in the definition. At first sight, it is not easy to see why one may want to consider operators of this type, but in fact, pseudo-ergodic operators are closely related to random operators. Under some reasonable conditions on the probability measure (see e.g.~\cite[Section 5.5.3]{HabilMarko}), one can show that a random operator is pseudo-ergodic almost surely. Therefore the definition of pseudo-ergodic operators is a nice circumvention of probabilistic arguments when dealing with random operators. We will make use of this fact for the rest of the paper and just mention here that every statement that holds for a pseudo-ergodic operator, holds for a random operator almost surely. The set of pseudo-ergodic operators is denoted by $\Psi E(U_n, \ldots, U_m)$. The notion of pseudo-ergodic operators goes back to Davies \cite{Davies}.

\subsection{Limit Operator Techniques}

Limit operators are an important tool in the study of band operators. For $k \in \mathbb{Z}$ define the $k$-th shift operator $V_k$ by $(V_k x)_j = x_{j-k}$ for all $x \in \Xb$. Let $A \in \Lc(\Xb)$ and let $h := (h_m)_{m \in \N}$ be a sequence of integers tending to infinity such that the strong limit\footnote{Sometimes different and more sophisticated notions of convergence are used to define limit operators. In the case of band operators on $\ell^2(\Z)$ all these notions coincide (see e.g.~\cite[Section 1.6.3]{Marko} or \cite[Example 4.6]{ChaLi}).} $A_h := \lim\limits_{m \to \infty} V_{-h_m}AV_{h_m}$ exists. Then $A_h$ is called a limit operator of $A$. The set of all limit operators is called the operator spectrum of $A$ and denoted by $\sigma^{\op}(A)$. Here are some basic properties of limit operators that we will need in the following (see e.g.~\cite[Proposition 3.4, Corollary 3.24]{Marko}):

\begin{prop} \label{basic}
Let $A,B \in \BO(\Xb)$ and let $h := (h_m)_{m \in \N}$ be a sequence of integers tending to infinity. Then the following statements hold:
\begin{itemize}
	\item There exists a subsequence $g := (g_m)_{m \in \N}$ of $h$ such that $A_g$ and $B_g$ exist.
  \item If $A_h$ and $B_h$ exist, so does $(A + B)_h$ and $(A+B)_h = A_h + B_h$.
  \item If $A_h$ and $B_h$ exist, so does $(AB)_h$ and $(AB)_h = A_h B_h$.
  \item If $A_h$ exists, so does $(A^*)_h$ and $(A^*)_h = (A_h)^*$.
  \item If $A_h$ exists, then $\left\|A_h\right\| \leq \left\|A\right\|$.
  \item If $A \in \Kc(\Xb)$, then $A_h = 0$.
\end{itemize}
\end{prop}

We call an operator $A$ Fredholm if $\ker(A)$ and $\im(A)^\perp$ are both finite-dimensional. As usual we define the spectrum
\[\spec(A) := \left\{\lambda \in \C : A - \lambda I \text{ is not invertible}\right\}\]
and the essential spectrum
\[\spess(A) := \left\{\lambda \in \C : A - \lambda I \text{ is not Fredholm}\right\}.\]

After introducing all the notation, we can cite the main theorem of limit operator theory (which holds in much more generality than stated and needed here).

\begin{thm} \label{mt}
(e.g.~\cite[Corollary 5.26]{HabilMarko})\\
Let $A \in \BO(\Xb)$. Then
\[\spess(A) = \bigcup\limits_{B \in \sigma^{\op}(A)} \spec(B).\]
\end{thm}

In order to apply this theorem to pseudo-ergodic operators, we use the following result that characterizes them in terms of limit operators.

\begin{prop} \label{lope}
Let $U_n, \ldots, U_m$ be non-empty and compact. Then $A \in \Psi E(U_n, \ldots, U_m)$ if and only if $\sigma^{\op}(A) = M(U_n, \ldots, U_m)$.
\end{prop}

\begin{proof}
For diagonal operators on $\ell^2(\Z)$ this is Corollary 3.70 in \cite{Marko}. The proof easily carries over to the case of band operators.
\end{proof}

Using this and $A \in M(U_n, \ldots, U_m)$, we get the following corollary.

\begin{cor} \label{mtpe}
Let $U_n, \ldots, U_m$ be non-empty and compact and let $A \in \Psi E(U_n, \ldots, U_m)$. Then
\begin{equation} \label{speceq}
\spec(A) = \spess(A) = \bigcup\limits_{B \in M(U_n, \ldots, U_m)} \spec(B).
\end{equation}
\end{cor}

In particular, we see that the spectrum of a pseudo-ergodic operator only depends on the sets $U_n, \ldots, U_m$. Furthermore, Equation \eqref{speceq} provides a somewhat easy method to obtain lower bounds for the spectrum of $A \in \Psi E(U_n, \ldots, U_m)$. Indeed, we can take any operator $B \in M(U_n, \ldots, U_m)$ with a known spectrum and get a lower bound for the spectrum of $A$. For example the spectrum of a periodic operator $B$ can be computed via the Fourier transform.

\begin{thm} \label{spcper}
(e.g.~\cite[Theorem 4.4.9]{Davies2})\\
Let $U_n, \ldots, U_m$ be non-empty and compact, $p \in \N$, $B \in M_{per,p}(U_n, \ldots, U_m)$ and let $B_k \in \Lc(\C^p)$ be defined by $(B_k)_{i,j} = B_{i+kp,j}$ for all $i,j \in \left\{1, \ldots, p\right\}$ and $k \in \Z$. Then
\begin{equation} \label{specper}
\spec(B) = \bigcup\limits_{\theta \in [0,2\pi)} \spec\left(\sum\limits_{k \in \Z} B_k e^{-ik\theta}\right).
\end{equation}
\end{thm}

This brief summary of limit operator theory is sufficient for the rest of this paper. We recommend \cite{Marko} and \cite{RaRoSi} for more details and further reading.

\section{The Numerical Range} \label{TheNumericalRange}

For the reader's convenience we start with the definition and some basic properties of the numerical range.

\begin{defn}
Let $A \in \Lc(\Xb)$. Then the numerical range is defined as
\[N(A) := \clos \left\{\left\langle Ax,x\right\rangle : x \in \Xb, \left\|x\right\| = 1\right\}.\]
For $\phi \in [0,2\pi)$ the (rotated) numerical abscissa is defined as
\[r_\phi(A) := \max\left\{\Re \, z : z \in N(e^{i\phi}A)\right\}.\]
\end{defn}

Note that the numerical range is usually defined without the closure (and denoted by $W(A)$), but we prefer to consider the numerical range as a compact set here. The following results are well-known and also hold in arbitrary Hilbert spaces.

\begin{thm} \label{HauToep}
(Hausdorff-Toeplitz)\\
Let $A \in \Lc(\Xb)$. Then $N(A)$ is convex.
\end{thm}

\begin{thm} \label{convspccN}
Let $A \in \Lc(\Xb)$. It holds
\[\conv(\spec(A)) \subseteq N(A)\]
with equality if $A$ is normal. Moreover,
\[\sup\limits_{\norm{x} = 1}\abs{\sp{Ax}{x}} \leq \norm{A}\]
with equality if $A$ is normal.
\end{thm}

To determine the numerical range of an operator $A$, one usually applies the following method by Johnson \cite{Johnson}. Since the numerical range is convex by Theorem \ref{HauToep}, it suffices to compute the numerical abscissae $r_{\phi}(A)$ for every angle $\phi \in [0,2\pi)$. Fix $\phi \in [0,2\pi)$ and let $B := \frac{1}{2}(e^{i\phi}A + e^{-i\phi}A^*)$. Then
\[r_{\phi}(A) = \sup\limits_{\norm{x} = 1} \Re \sp{e^{i\phi}Ax}{x} = \sup\limits_{\norm{x} = 1} \frac{1}{2}\sp{(e^{i\phi}A + e^{-i\phi}A^*)x}{x} = \sup\limits_{\norm{x} = 1} \sp{Bx}{x} = r_0(B).\]
Since $B$ is self-adjoint, $r_{\phi}(A)$ is exactly equal to the rightmost point of the spectrum of $B$. This observation is the starting point for almost every result we prove in this paper.

We will also find it useful to talk about convergence of set sequences.

\begin{defn}
Let $(M_n)_{n \in \N}$ be a sequence of compact subsets of $\C$. Then we define
\begin{align*}
\limsup\limits_{n \to \infty} M_n &:= \set{m \in \C: m \text{ is an accumulation point of a sequence } (m_n)_{n \in \N}, m_n \in M_n},\\
\liminf\limits_{n \to \infty} M_n &:= \set{m \in \C: m \text{ is the limit of a sequence } (m_n)_{n \in \N}, m_n \in M_n}.
\end{align*}
The Hausdorff metric for compact sets $A,B \subset \C$ is defined as
\[h(A,B) := \max\set{\max\limits_{a \in A} \min_{b \in B} \abs{a-b},\max\limits_{b \in B} \min_{a \in A} \abs{a-b}}.\]
Moreover, we define $\lim\limits_{n \to \infty} M_n$ as the limit of the sequence $(M_n)_{n \in \N}$ w.r.t.~the Hausdorff metric.
\end{defn}

These notions are compatible with each other in the sense that they satisfy the same relations as they do for ordinary sequences:

\begin{prop}
(\cite[Proposition 3.6]{HaRoSi})\\
Let $(M_n)_{n \in \N}$ be a sequence of compact subsets of $\C$. Then the limit $\lim\limits_{n \to \infty} M_n$ exists if and only if $\limsup\limits_{n \to \infty} M_n = \liminf\limits_{n \to \infty} M_n$ and in this case we have
\[\lim\limits_{n \to \infty} M_n = \limsup\limits_{n \to \infty} M_n = \liminf\limits_{n \to \infty} M_n.\]
\end{prop}

\subsection{Limit Operator and Approximation Results}

We will first prove the following limit operator result, which can be proven (without further effort) in much more generality than we state it here.

\begin{thm} \label{mtnr}
Let $A \in \BO(\Xb)$. Then
\begin{equation} \label{mtnreq}
\bigcap\limits_{K \in \Kc(\Xb)} N(A+K) = \conv\left(\bigcup\limits_{B \in \sigma^{\op}(A)} N(B)\right).
\end{equation}
\end{thm}

To prove this, we need the following lemma that we will then apply to sequences $(V_{-h_n}(A+K)V_{h_n})_{n \in \N}$, where $K \in \Kc(\Xb)$ and $(h_n)_{n \in \N}$ is a sequence of integers tending to infinity.

\begin{lem} \label{nrBS}
Let $A \in \Lc(\Xb)$ and let $(A_n)_{n \in \N}$ be a sequence in $\Lc(\Xb)$ that converges to $A$ in weak operator topology. Then $N(A) \subseteq \liminf\limits_{n \to \infty} N(A_n)$.
\end{lem}

\begin{proof}
$A_n \to A$ in the weak operator topology implies $\left\langle (A_n-A)x,x \right\rangle \to 0$ for all $x \in \Xb$ as $n \to \infty$. Let $z \in N(A)$. Choose $x_1 \in \Xb$ with $\left\|x_1\right\| = 1$ such that $\left|z - \left\langle Ax_1,x_1 \right\rangle\right| < 1$ and $n_1$ such that $\left|\left\langle (A_n-A)x_1,x_1 \right\rangle\right| < 1$ for all $n \geq n_1$. For $j \in \N$ choose $x_{j+1} \in \Xb$ with $\left\|x_{j+1}\right\| = 1$ such that $\left|z - \left\langle Ax_{j+1},x_{j+1} \right\rangle\right| < \frac{1}{j+1}$ and $n_{j+1} > n_j$ such that $\left|\left\langle (A_n-A)x_{j+1},x_{j+1} \right\rangle\right| < \frac{1}{j+1}$ for all $n \geq n_{j+1}$. Of course this implies $\left|z - \left\langle A_nx_j,x_j \right\rangle\right| < \frac{2}{j}$ for all $n \geq n_j$. Now define a sequence $(z_n)_{n \in \N}$ of complex numbers as follows. For $n < n_1$ choose $z_n \in N(A_n)$ arbitrarily. For $j \in \N$ and $n_j \leq n < n_{j+1}$ choose $z_n \in N(A_n)$ such that $\left|z - z_n\right| < \frac{2}{j}$. We get $\left|z - z_n\right| \to 0$ as $n \to \infty$. Thus $N(A) \subseteq \liminf\limits_{n \to \infty} N(A_n)$.
\end{proof}

\begin{proof}[Proof of Theorem \ref{mtnr}]
Let $B \in \sigma^{\op}(A)$ and $K \in \Kc(\Xb)$. To prove ``$\supseteq$'' it suffices to show $N(B) \subseteq N(A+K)$ because the intersection of convex sets is again convex. So let $h$ be a sequence of integers tending to infinity such that $A_h = B$. By Proposition \ref{basic}, $B$ is also a limit operator of $A+K$:
\[(A+K)_h = A_h + K_h = A_h + 0 = A_h = B.\]
Applying Lemma \ref{nrBS} to the sequence $(V_{-h_n}(A+K)V_{h_n})_{n \in \N}$ and using that the numerical range is invariant under unitary transformations, we get
\[N(B) \subseteq \liminf\limits_{n \to \infty} N(V_{-h_n}(A+K)V_{h_n}) = \liminf\limits_{n \to \infty} N(A+K) = N(A+K).\]

To prove the other inclusion, recall that it suffices to compare numerical abscissae, i.e.~to show
\[\inf\limits_{K \in \Kc(\Xb)} r_\phi(A+K) \leq \max\set{r_\phi(B) : B \in \sigma^{\op}(A)}.\]
for all $\phi \in [0,2\pi)$. Since $r_\phi(A) = r_0(e^{i\phi}A)$ for all $A \in \Lc(\Xb)$ and $\phi \in [0,2\pi)$, it even suffices to consider $\phi = 0$. Set $z_0 := \norm{A}$. Then
\begin{align*}
r_0(A+K) &= \sup\limits_{\norm{x} = 1} \Re\sp{(A+K)x}{x}\\
&= \sup\limits_{\norm{x} = 1} \Re\sp{(A + K + z_0 I)x}{x} - z_0\\
&\leq \sup\limits_{\norm{x} = 1} \abs{\Re\sp{(A + K + z_0 I)x}{x}} - z_0\\
&= \sup\limits_{\norm{x} = 1} \abs{\frac{1}{2}\sp{(A+K + (A+K)^* + 2z_0 I)x}{x}} - z_0\\
&= \frac{1}{2}\norm{A+K + (A+K)^* + 2z_0 I} - z_0,
\end{align*}
where we applied Theorem \ref{convspccN} to the self-adjoint (hence normal) operator $A+K + (A+K)^* + 2z_0 I$. Taking the infimum, we arrive at
\begin{align*}
\inf\limits_{K \in \Kc(\Xb)} r_0(A+K) &\leq \frac{1}{2}\inf\limits_{K \in \Kc(\Xb)}\norm{A+K + (A+K)^* + 2z_0 I} - z_0\\
&= \frac{1}{2}\inf\limits_{\substack{K \in \Kc(\Xb) \\ K = K^*}}\norm{A + A^* + K + 2z_0 I} - z_0.\\
\end{align*}
For a self-adjoint operator $C \in \Lc(\Xb)$, the norm $\norm{C+K}$ is minimized by a self-adjoint operator $K \in \Kc(\Xb)$. This can be seen as follows:
\begin{align*}
\norm{C+K} &\geq \sup\limits_{\norm{x} = 1} \abs{\sp{(C+K)x}{x}}\\
&= \sup\limits_{\norm{x} = 1} \abs{\sp{\left(C + \frac{K+K^*}{2}\right)x}{x} + \sp{\left(C + \frac{K-K^*}{2}\right)x}{x}}\\
&\geq \sup\limits_{\norm{x} = 1} \abs{\sp{\left(C + \frac{K+K^*}{2}\right)x}{x}}\\
&= \norm{C + \frac{K+K^*}{2}},
\end{align*}
where we used Theorem \ref{convspccN} and the fact that $\sp{\left(C + \frac{K+K^*}{2}\right)x}{x} \in \R$ and $\sp{\left(C + \frac{K-K^*}{2}\right)x}{x} \in i\R$ for all $x \in \Xb$. Moreover, we have
\[\inf\limits_{K \in \Kc(\Xb)} \norm{A+K} = \max\limits_{B \in \sigma^{\op}(A)} \norm{B}\]
for all $A \in \BO(\Xb)$ by \cite[Theorem 3.2]{HaLiSe}. Combining these results and using Proposition \ref{basic}, we get
\begin{align*}
\inf\limits_{K \in \Kc(\Xb)} r_0(A+K) &\leq \frac{1}{2} \inf\limits_{K \in \Kc(\Xb)} \norm{A + A^* + K + 2z_0 I} - z_0\\
&= \frac{1}{2} \max\set{\norm{B} : B \in \sigma^{\op}(A + A^* + 2z_0 I)} - z_0\\
&= \frac{1}{2} \max\set{\norm{B + B^* + 2z_0 I} : B \in \sigma^{\op}(A)} - z_0.
\end{align*}
Since $r_{\phi}(B) \leq \norm{B} \leq \norm{A}$ for all $\phi \in [0,2\pi)$ by Theorem \ref{convspccN} and Proposition \ref{basic}, $N(B + z_0 I)$ is contained in the right half plane for every $B \in \sigma^{\op}(A)$. This implies
\begin{align*}
r_0(B) &= \sup\limits_{\norm{x} = 1} \Re\sp{Bx}{x}\\
&= \sup\limits_{\norm{x} = 1} \Re\sp{(B + z_0 I)x}{x} - z_0\\
&= \sup\limits_{\norm{x} = 1} \abs{\Re\sp{(B + z_0 I)x}{x}} - z_0\\
&= \sup\limits_{\norm{x} = 1} \abs{\frac{1}{2}\sp{(B + B^* + 2z_0 I)x}{x}} - z_0\\
&= \frac{1}{2}\norm{B + B^* + 2z_0 I} - z_0.
\end{align*}
We conclude
\[\inf\limits_{K \in \Kc(\Xb)} r_0(A+K) \leq \max\set{r_0(B) : B \in \sigma^{\op}(A)}.\qedhere\]
\end{proof}

If we apply this result to pseudo-ergodic operators, we get the following corollary:

\begin{cor} \label{mtnrpe}
Let $U_n, \ldots, U_m$ be non-empty and compact. It holds
\[N(A) = \bigcup\limits_{B \in M(U_n, \ldots, U_m)} N(B)\]
for all $A \in \Psi E(U_n, \ldots, U_m)$.
\end{cor}

Note that taking the convex hull is obviously not necessary here. In fact, it suffices to consider periodic operators on the right-hand side:

\begin{cor} \label{nrperapprox}
Let $U_n, \ldots, U_m$ be non-empty and compact. It holds
\[N(A) = \clos\left(\bigcup\limits_{B \in M_{per}(U_n, \ldots, U_m)} N(B)\right)\]
for all $A \in \Psi E(U_n, \ldots, U_m)$.
\end{cor}

\begin{proof}
Let $A \in \Psi E(U_n, \ldots, U_m)$. It is not difficult to find a sequence $(A_k)_{k \in \N} \subset M_{per}(U_n, \ldots, U_m)$ that converges weakly to $A$ (even strongly). Thus by Lemma \ref{nrBS} and Corollary \ref{mtnrpe}, we have
\[N(A) \subseteq \liminf\limits_{k \to \infty} N(A_k) \subseteq \clos\left(\bigcup\limits_{B \in M_{per}(U_n, \ldots, U_m)} N(B)\right) \subseteq N(A).\qedhere\]
\end{proof}

In the next section we will see that in the case of a tridiagonal pseudo-ergodic operator $A$, it even suffices to consider the Laurent operators contained in $\sigma^{\op}(A)$.

So far we only considered numerical ranges of operators $A \in \BO(\ell^2(\Z))$. However, it is sometimes more convenient to work with operators $A \in \Lc(\ell^2(\N))$. We will thus find the following well-known proposition useful.

\begin{prop} \label{selfsim}
Let $A \in \BO(\ell^2(\Z))$ and let $A_+ := P_{\N}AP_{\N}|_{\im P_{\N}} \in \Lc(\ell^2(\N))$, where $P_{\N}$ denotes the projection onto $\text{span}\left\{e_1,e_2, \ldots \right\}$. If there exists a sequence $(h_m)_{m \in \N}$ of integers tending $+\infty$ such that $A_h$ exists and is equal to $A$, then $N(A) = N(A_+)$.
\end{prop}

\begin{proof}
Clearly, $\sp{A_+x}{x} = \sp{Ax}{x}$ for all $x \in \im P_{\N}$ and thus $N(A_+) \subseteq N(A)$. Conversely, let $c \in N(A_+)$, $Q_{\N} := I - P_{\N}$ and consider $\tilde{A} := P_{\N}AP_{\N}|_{\im P_{\N}} + cQ_{\N}|_{\im Q_{\N}}$. Then
\[\langle\tilde{A}x,x\rangle = \sp{A_+P_{\N}x}{P_{\N}x} + c\sp{Q_{\N}x}{Q_{\N}x} = \sp{A_+\frac{P_{\N}x}{\norm{P_{\N}x}}}{\frac{P_{\N}x}{\norm{P_{\N}x}}}\norm{P_{\N}x}^2 + c\norm{Q_{\N}x}^2.\]
Since $\norm{P_{\N}x}^2 + \norm{Q_{\N}x}^2 = \norm{x}^2$ and $N(A_+)$ is convex, we get $N(\tilde{A}) \subseteq N(A_+)$. Moreover, $A$ is a limit operator of $\tilde{A}$ and thus $N(A) \subseteq N(\tilde{A})$ by Theorem \ref{mtnr}. We conclude $N(A) = N(A_+)$.
\end{proof}

\subsection{Tridiagonal Pseudo-Ergodic Operators}

In this section we focus on the case of tridiagonal pseudo-ergodic operators. Here the following simplification of Corollary \ref{nrperapprox} can be achieved:

\begin{thm} \label{ntridiag}
Let $U_{-1}$, $U_0$ and $U_1$ be non-empty and compact. Then for $A \in \Psi E(U_{-1},U_0,U_1)$ the following formula holds:
\[N(A) \stackrel{(i)}{=} \conv\left(\bigcup\limits_{B \in L(U_{-1},U_0,U_1)} \spec(B)\right) \stackrel{(ii)}{=}  \conv\left(\bigcup\limits_{\substack{u_k \in U_k, \\ k = -1,0,1}} \left\{u_{-1}e^{i\theta} + u_0 + u_1e^{-i\theta} : \theta \in [0,2\pi)\right\}\right).\]
In particular, $\spec(A) = N(A)$ if $\bigcup\limits_{B \in L(U_{-1},U_0,U_1)} \spec(B)$ is convex.
\end{thm}

\begin{proof}
The last assertion follows from $(i)$ since
\[\bigcup\limits_{B \in L(U_{-1},U_0,U_1)} \spec(B) \subseteq \spec(A) \subseteq N(A)\]
by Corollary \ref{mtpe} and Theorem \ref{convspccN}. Moreover, $(ii)$ follows immediately from Theorem \ref{spcper}. We thus focus on the proof of $(i)$.

``$\supseteq$'' : Theorem \ref{convspccN} and  Corollary \ref{mtpe} imply
\[N(A) \supseteq \spec(A) \supseteq \bigcup\limits_{B \in L(U_{-1},U_0,U_1)} \spec(B)\]
as above. Taking the convex hull on both sides yields
\[N(A) \supseteq \conv\left(\bigcup\limits_{B \in L(U_{-1},U_0,U_1)} \spec(B)\right).\]
by Theorem \ref{HauToep}.

``$\subseteq$'': As in the proof of Theorem \ref{mtnr}, it suffices to compare $r_0(A)$ with $\max\limits_{B \in L(U_{-1},U_0,U_1)} r_0(B)$. This then implies
\[N(A) \subseteq \conv\left(\bigcup\limits_{B \in L(U_{-1},U_0,U_1)} N(B)\right)\]
and hence
\[N(A) \subseteq \conv\left(\bigcup\limits_{B \in L(U_{-1},U_0,U_1)} \conv(\spec(B))\right) = \conv\left(\bigcup\limits_{B \in L(U_{-1},U_0,U_1)} \spec(B)\right)\]
because Laurent operators are normal. We also set $z_0 = \norm{A}$ again, which implies $N(B + z_0 I) \subset \C_{\Re \geq 0}$ for all $B \in M(U_{-1},U_0,U_1)$. It follows
\begin{align*}
r_0(A) &= \sup\limits_{\left\|x\right\| = 1} \Re \, \left\langle Ax,x \right\rangle\\
&= \sup\limits_{\left\|x\right\| = 1} \Re \, \left\langle (A + z_0 I)x,x \right\rangle - z_0\\
&= \sup\limits_{\left\|x\right\| = 1} \left|\Re \, \left\langle (A + z_0 I)x,x \right\rangle\right| - z_0\\
&= \frac{1}{2} \sup\limits_{\left\|x\right\| = 1} \left|\left\langle (A + A^* + 2z_0 I)x,x \right\rangle\right| - z_0\\
&= \frac{1}{2} \left\|A + A^* + 2z_0 I\right\| - z_0,
\end{align*}
where we used Theorem \ref{convspccN} in the last line. Using that the norm of an operator is bounded by the sum of the maximal elements of its diagonals (also called Wiener estimate, see e.g.~\cite[p.~25]{Marko}), we arrive at
\begin{align} \label{ineq1}
r_0(A) &= \frac{1}{2} \left\|A + A^* + 2z_0 I\right\| - z_0\notag\\
&\leq \max\limits_{\substack{u_{-1} \in U_{-1} \\ u_1 \in U_1}} \left|u_1 + \overline{u_{-1}}\right| + \frac{1}{2} \max\limits_{u_0 \in U_0} \left|u_0 + \overline{u_0} + 2z_0\right| - z_0.
\end{align}
Fix $w_{-1} \in U_{-1}$, $w_0 \in U_0$ and $w_1 \in U_1$ such that the maximum in \eqref{ineq1} is attained, i.e.
\[\max\limits_{\substack{u_{-1} \in U_{-1} \\ u_1 \in U_1}} \left|e^{i\phi}u_1 + e^{-i\phi}\overline{u_{-1}}\right| = \left|e^{i\phi}w_1 + e^{-i\phi}\overline{w_{-1}}\right|\]
and
\[\max\limits_{u_0 \in U_0} \left|e^{i\phi}u_0 + e^{-i\phi}\overline{u_0} + 2z_0\right| = \left|e^{i\phi}w_0 + e^{-i\phi}\overline{w_0} + 2z_0\right|.\]
It is not hard to see that the spectrum of a tridiagonal Laurent operator $L(v_{-1},v_0,v_1)$ (to simplify the notation we identify the set $L(v_{-1},v_0,v_1) := L(\left\{v_{-1}\right\},\left\{v_0\right\},\left\{v_1\right\})$ with its only element) is given by an ellipse with center $v_0$ and half-axes $\bigl\lvert\lvert v_{-1}\rvert \pm\lvert v_1\rvert\bigr\rvert$ (see e.g. \cite{LiRo}). If in addition $C := L(v_{-1},v_0,v_1)$ is self-adjoint, then its spectrum is given by the interval
\[\spec(C) = [v_0 - \abs{v_{-1}} - \abs{v_1},v_0 + \abs{v_{-1}} + \abs{v_1}]\]
and thus $\left\|C\right\| = \left|v_0\right| + \left|v_{-1}\right| + \left|v_1\right|$. In our case, if we put $B := L(w_{-1},w_0,w_1)$, we get
\[\left\|B + B^* + 2z_0 I\right\| = 2\left|w_1 + \overline{w_{-1}}\right| + \left|w_0 + \overline{w_0} + 2z_0\right|\]
and therefore
\[r_0(A) \leq \frac{1}{2}\left\|B + B^* + 2z_0 I\right\| - z_0.\]
From here we can go all the way back to finish the proof:
\begin{align*}
r_0(A) &\leq \frac{1}{2} \left\|B + B^* + 2z_0 I\right\| - z_0\\
&= \frac{1}{2} \sup\limits_{\left\|x\right\| = 1} \left|\left\langle (B + B^* + 2z_0 I)x,x \right\rangle\right| - z_0\\
&= \sup\limits_{\left\|x\right\| = 1} \left|\Re \, \left\langle (B + z_0 I)x,x \right\rangle\right| - z_0\\
&= \sup\limits_{\left\|x\right\| = 1} \Re \, \left\langle (B + z_0 I)x,x \right\rangle - z_0\\
&= r_0(B).\qedhere
\end{align*}
\end{proof}

Combining Corollary \ref{mtpe}, Theorem \ref{convspccN} and Theorem \ref{ntridiag} we also get the following corollary.

\begin{cor}
Let $U_{-1}$, $U_0$ and $U_1$ be non-empty and compact and let $A \in \Psi E(U_{-1},U_0,U_{1})$. Then $A$ has the following property:
\[N(A) = \conv(\spec(A)).\]
\end{cor}

This corollary is quite remarkable because one can usually not expect this property from non-normal operators. As a consequence, any tridiagonal random operator has this property almost surely. We do not know if pseudo-ergodic operators with more than three diagonals share this property, but we do know that Theorem \ref{ntridiag} is wrong if the tridiagonality assumption is dropped.

\begin{ex} \label{ex}
Let $U_{-2} = \left\{1\right\}$, $U_{-1} = \left\{\pm 1\right\}$, $U_0 = \left\{0\right\}$, $U_1 = \left\{1\right\}$ and $U_2 = \left\{1\right\}$. Consider the $3$-periodic operator
\[A = \begin{pmatrix} \ddots & \ddots & \ddots & & & \\ \ddots & 0 & 1 & 1 & &\\ \ddots & 1 & 0 & 1 & 1 & \\ & 1 & 1 & 0 & -1 & \ddots \\ & & 1 & 1 & 0 & \ddots \\ & & & \ddots & \ddots & \ddots \end{pmatrix} \in M_{per,3}(U_{-2}, \ldots, U_2).\]
Then
\[B := \frac{1}{2}(A + A^*) = \begin{pmatrix} \ddots & \ddots & \ddots & & & \\ \ddots & 0 & 1 & 1 & &\\ \ddots & 1 & 0 & 1 & 1 & \\ & 1 & 1 & 0 & 0 & \ddots \\ & & 1 & 0 & 0 & \ddots \\ & & & \ddots & \ddots & \ddots \end{pmatrix}.\]
By Theorem \ref{spcper} we get 
\[\spec(B) = \bigcup\limits_{\theta \in [0,2\pi)} \spec(b(\theta)),\]
where
\[b(\theta) := \begin{pmatrix} 0 & 1 + e^{-i\theta} & 1 + e^{-i\theta} \\ 1 + e^{i\theta} & 0 & e^{-i\theta} \\ 1 + e^{i\theta} & e^{i\theta} & 0 \end{pmatrix}.\]
The spectrum of $b(0)$ is given by $\set{\frac{1}{2} - \frac{\sqrt{33}}{2},-1,\frac{1}{2} + \frac{\sqrt{33}}{2}}$. So in particular, we get $\frac{1}{2} - \frac{\sqrt{33}}{2} \in \spec(B)$ and thus
\[r_{\pi}(A) = r_0(-A) \geq \frac{\sqrt{33}}{2} - \frac{1}{2} > \frac{9}{4}.\]
Let us denote the two operators in $L(U_{-2}, \ldots, U_2)$ by $C_1$ and $C_2$. We get
\[\min\limits_{z \in \spec(C_1)} \Re \, z = \min\limits_{\theta \in [0,2\pi)} \Re\left(e^{-2i\theta} + e^{-i\theta} + e^{i\theta} + e^{2i\theta}\right) = \min\limits_{\theta \in [0,2\pi)} 2(\cos(2\theta) + \cos(\theta)) = -\frac{9}{4}\]
and
\[\min\limits_{z \in \spec(C_2)} \Re \, z = \min\limits_{\theta \in [0,2\pi)} \Re\left(e^{-2i\theta} + e^{-i\theta} - e^{i\theta} + e^{2i\theta}\right) = \min\limits_{\theta \in [0,2\pi)} 2\cos(2\theta) = -2 \]
by Theorem \ref{spcper} again. This implies that the numerical range of $A$ exceeds the convex hull of the spectra of Laurent operators in the direction of the negative real axis, i.e.
\[N(A) \not\subseteq \conv\left(\bigcup\limits_{B \in L(U_{-1},U_0,U_1)} \spec(B)\right).\]
So in particular, in view of Corollary \ref{mtnrpe}, Theorem \ref{ntridiag} is not valid for five diagonals.
\end{ex}

\subsection{A Method to Compute Numerical Ranges for General Tridiagonal Operators} \label{method}

In this section we introduce a method to compute numerical ranges for tridiagonal operators. As explained at the beginning of Section \ref{TheNumericalRange}, it suffices to compute the numerical abscissae $r_{\phi}$ for $\phi \in [0,2\pi)$. Fix $\phi \in [0,2\pi)$ and recall that we have $r_{\phi}(A) = r_0(B)$ for $B := \frac{1}{2}(e^{i\phi}A + e^{-i\phi}A^*)$. In case $A$ is a tridiagonal infinite matrix acting on $\ell^2(\N)$ or $\ell^2(\Z)$, the non-zero entries of $B$ are given by
\begin{align*}
B_{j,j-1} &= \frac{1}{2}(e^{i\phi}A_{j,j-1} + e^{-i\phi}\overline{A_{j-1,j}}),\\
B_{j,j} &= \frac{1}{2}(e^{i\phi}A_{j,j} + e^{-i\phi}\overline{A_{j,j}}) = \Re(e^{i\phi}A_{j,j}),\\
B_{j,j+1} &= \frac{1}{2}(e^{i\phi}A_{j,j+1} + e^{-i\phi}\overline{A_{j+1,j}})\\
\end{align*}
for all $j$ in the respective index set. $B$ can now be transformed to a real symmetric matrix by applying the unitary diagonal transformation that is defined recursively as follows:
\begin{align*}
T_{1,1} &= 1,\\
T_{j+1,j+1} &= \sign(B_{j,j+1})T_{j,j}
\end{align*}
for all $j$ in the respective index set, where $\sign \from \C \to \T$ is defined as
\[\sign(z) := \begin{cases} \frac{z}{\abs{z}} & \text{if } z \neq 0, \\ 1 & \text{if } z = 0.\end{cases}\]
$C := TBT^*$ is then real and symmetric with $r_0(C) = r_0(B) = r_{\phi}(A)$ and
\begin{align} \label{C}
C_{j,j} &= \Re(e^{i\phi}A_{j,j}) \in \R,\notag\\
C_{j,j+1} &= \frac{1}{2}\abs{e^{i\phi}A_{j,j+1} + e^{-i\phi}\overline{A_{j+1,j}}} \geq 0.
\end{align}
Thus the computation of $r_{\phi}(A)$ is reduced to the computation of $r_0(C)$, which is also the rightmost point in the spectrum of $C$. In the following we can also assume that $C_{j,j+1} > 0$ for all $j$ because if $C_{j,j+1} = 0$ for some $j$, then $C$ can be divided into blocks and the spectrum of $C$ is then given by the closure of the union of the spectra of these blocks. Moreover, shifting $C$ by $\lambda I$ for some $\lambda \in \R$ only shifts the spectrum of $C$ by $\lambda$. Thus we can also assume that $C$ only has positive entries on its main diagonal.

This matrix $C$ now satisfies the requirements of the following lemma by Szwarc\footnote{Szwarc \cite{Szwarc} actually proved it for $C \in \Lc(\ell^2(\Z))$, but the proof is very similar for $C \in \Lc(\ell^2(\N))$.} that is basically a reformulation of the Schur test.

\begin{lem} \label{SzwarcLemma1}
(\cite[Proposition 1]{Szwarc})\\
Let $C \in \Lc(\ell^2(\N))$ be real, symmetric and tridiagonal with $C_{j,j}, C_{j,j+1} > 0$ for all $j \in \N$ and $N > \sup\limits_{j \in \N} C_{j,j}$. If there is a sequence $(g_j)_{j \in \N}$ that satisfies $g_j \in [0,1]$ and
\begin{equation} \label{BoundCondition}
\frac{C_{j,j+1}^2}{(N - C_{j,j})(N - C_{j+1,j+1})} \leq g_{j+1}(1 - g_j)
\end{equation}
for all $j \in \N$, then $r_0(C) \leq N$.
\end{lem}

In fact, also the converse is true:

\begin{lem} \label{SzwarcLemma2}
(\cite[Proposition 2]{Szwarc})\\
Let $C \in \Lc(\ell^2(\N))$ be real, symmetric and tridiagonal with $C_{j,j}, C_{j,j+1} > 0$ for all $j \in \N$. Then there exists a sequence $(g_j)_{j \in \N}$ with the following properties:
\begin{itemize}
	\item $g_j \in [0,1)$ for all $j \in \N$,
	\item $g_j = 0$ if and only if $j = 1$,
	\item the following equality holds for all $j \in \N$:
	\begin{equation} \label{BoundProperty}
	\frac{C_{j,j+1}^2}{(r_0(C) - C_{j,j})(r_0(C) - C_{j+1,j+1})} = g_{j+1}(1 - g_j).
	\end{equation}
\end{itemize}
\end{lem}

To demonstrate the procedure, we prove the following proposition that we need later on. Note that this proposition can also be shown using Theorem \ref{spcper} applied to $B := \frac{1}{2}(A + A^*)$. This results in the computation of the eigenvalues of a $2 \times 2$ matrix and one obtains Corollary \ref{2percor} directly.

\begin{prop} \label{2per}
Let $\I \in \set{\N,\Z}$, let $A \in \Lc(\ell^2(\I))$ be tridiagonal and $2$-periodic and let $N > \sup\limits_{i \in \I} \Re \, A_{i,i}$. Further assume that $A + A^*$ is not diagonal. Define 
\[\eta_1(A) := \frac{\left|A_{1,2} + \overline{A_{2,1}}\right|^2}{4(N - \Re \, A_{1,1})(N - \Re \, A_{2,2})}, \qquad \eta_2(A) := \frac{\left|A_{2,3} + \overline{A_{3,2}}\right|^2}{4(N- \Re \, A_{2,2})(N - \Re \, A_{3,3})}.\]
Then we have $\sqrt{\eta_1(A)} + \sqrt{\eta_2(A)} = 1$ if and only if $N = r_0(A)$.
\end{prop}

\begin{proof}
Clearly, $A \in \Lc(\ell^2(\Z))$ and $P_{\N}AP_{\N}|_{P_{\N}} \in \Lc(\ell^2(\N))$ have the same numerical range by Proposition \ref{selfsim}. It thus suffices to consider the case $A \in \Lc(\ell^2(\N))$. Let $C$ be as in $\eqref{C}$ with $\phi = 0$ so that $r_0(A) = r_0(C)$. We can assume that $C_{j,j} > 0$ for all $j \in \N$ (shifting by $\lambda \in \R$ does not change anything).

If $A_{1,2} + \overline{A_{2,1}} = 0$. Then $\eta_1(A) = 0$ and an easy computation shows $\eta_2(A) = 1$ if and only if $N = r_0(A)$. The case $A_{2,3} + \overline{A_{3,2}} = 0$ is similar. So let us assume $A_{j,j+1} + \overline{A_{j+1,j}} \neq 0$ for all $j \in \N$ for the rest of the proof. Clearly, this implies $\eta_1(A), \eta_2(A) > 0$ and $C_{j,j+1} > 0$ for all $j \in \N$.

Let $N = r_0(A)$. Lemma \ref{SzwarcLemma2} applied to $C$ yields a sequence $(g_j)_{j \in \N}$ with the properties
\begin{itemize}
	\item $g_j \in [0,1)$ for all $j \in \N$
	\item $g_j = 0$ if and only if $j = 1$
	\item the following equality holds for all $j \in \N$:
	\[\frac{\abs{A_{j,j+1} + \overline{A_{j+1,j}}}^2}{4(r_0(A) - \Re A_{j,j})(r_0(A) - \Re A_{j+1,j+1})} = g_{j+1}(1 - g_j).\]
\end{itemize}
Since $A$ is $2$-periodic, we have
\begin{align*}
&\eta_1(A) = \frac{\abs{A_{1,2} + \overline{A_{2,1}}}^2}{4(r_0(A) - \Re \, A_{1,1})(r_0(A) - \Re \, A_{2,2})} = g_2\\
&\eta_2(A) = \frac{\abs{A_{2,3} + \overline{A_{3,2}}}^2}{4(r_0(A) - \Re \, A_{2,2})(r_0(A) - \Re \, A_{1,1})} = g_3(1-g_2)\\
&\eta_1(A) = \frac{\abs{A_{1,2} + \overline{A_{2,1}}}^2}{4(r_0(A) - \Re \, A_{1,1})(r_0(A) - \Re \, A_{2,2})} = g_4(1-g_3)\\
&\quad \vdots \qquad \qquad \qquad \qquad \qquad \qquad \qquad \qquad \quad \vdots
\end{align*}
We observe $\eta_1(A) = g_2 \in (0,1)$ and $\eta_2(A) = g_3(1-g_2) \in (0,1)$. If $j$ is odd, we deduce the following recursion:
\begin{equation} \label{OddRecursion}
g_{j+2} = \frac{\eta_2(A)}{1-g_{j+1}} = \frac{\eta_2(A)}{1-\frac{\eta_1(A)}{1 - g_j}} = \frac{(1 - g_j)\eta_2(A)}{1 - g_j - \eta_1(A)}.
\end{equation}
The corresponding iteration function
\begin{align} \label{IterationFunction}
&f: (0,1 - \eta_1(A)) \to \R,\notag\\
&x \mapsto \frac{(1 - x)\eta_2(A)}{1 - x - \eta_1(A)}
\end{align}
has a positive derivative
\begin{equation} \label{df/dx}
\frac{\text{d}}{\text{d}x} \frac{(1 - x)\eta_2(A)}{1 - x - \eta_1(A)} = \frac{\eta_1(A)\eta_2(A)}{(1 - x - \eta_1(A))^2} > 0
\end{equation}
since $\eta_1(A),\eta_2(A) > 0$. Thus $f$ is strictly increasing. Since $(g_j)_{j \in 2\N-1}$ is a sequence in $[0,1)$, it is in fact a sequence in $[0,1-\eta_1(A))$. Indeed, if $g_j \geq 1-\eta_1(A)$, then by Equation \eqref{OddRecursion}, $g_{j+2}$ is either not defined or negative, a contradiction. Moreover, we have
\[g_3 = \frac{\eta_2(A)}{1-\eta_1(A)} > 0 = g_1\]
since $\eta_1(A),\eta_2(A) \in (0,1)$. We conclude that $(g_j)_{j \in 2\N-1}$ is strictly increasing, hence convergent. Denote the limit of this sequence by $x^*$. By the fixed-point theorem, $x^*$ has to be a fixed point of the iteration function $f$. After some rearranging, we get two possible candidates for a fixed point:
\begin{align} \label{FixedPoint}
\frac{(1 - x^*)\eta_2(A)}{1 - x^* - \eta_1(A)} = x^* \quad &\Leftrightarrow \quad (1 - x^*)\eta_2(A) = x^*(1 - x^* - \eta_1(A))\notag\\
&\Leftrightarrow \quad (x^*)^2 - (1 + \eta_2(A) - \eta_1(A))x^* + \eta_2(A) = 0\notag\\
&\Leftrightarrow \quad x^* = \frac{1 + \eta_2(A) - \eta_1(A) \pm \sqrt{(1 + \eta_2(A) - \eta_1(A))^2 - 4\eta_2(A)}}{2}.
\end{align}
Of course the fixed point we are looking for has to be real and thus $(1 + \eta_2(A) - \eta_1(A))^2 - 4\eta_2(A)$ has to be non-negative. It follows
\begin{align*}
0 &\leq (1 + \eta_2(A) - \eta_1(A))^2 - 4\eta_2(A)\\
&= 1 + \eta_2(A)^2 + \eta_1(A)^2 + 2\eta_2(A) - 2\eta_1(A) - 2\eta_1(A)\eta_2(A) - 4\eta_2(A)\\
&= \eta_2(A)^2 - 2(1 + \eta_1(A))\eta_2(A) + (1 - \eta_1(A))^2.
\end{align*}
Solving for $\eta_2(A)$ yields
\[\eta_2(A) \leq 1 + \eta_1(A) - \sqrt{(1 + \eta_1(A))^2 - (1 - \eta_1(A))^2} = 1 + \eta_1(A) - 2\sqrt{\eta_1(A)} = (1 - \sqrt{\eta_1(A)})^2,\]
since $\eta_2(A) < 1$. This inequality now implies $\sqrt{\eta_1(A)} + \sqrt{\eta_2(A)} \leq 1$. As we will prove later, this inequality is actually an equality.

Conversely, let $\sqrt{\eta_1(A)} + \sqrt{\eta_2(A)} = 1$. Of course, we can again assume that $\I = \N$. Define the sequence $(g_j)_{j \in \N}$ as follows:
\begin{align*}
g_1 &:= 0,\\
g_{j+1} &:= \frac{\eta_1(A)}{1 - g_j} \quad \text{if $j$ is odd},\\
g_{j+1} &:= \frac{\eta_2(A)}{1 - g_j} \quad \text{if $j$ is even}.
\end{align*}
In order to apply Lemma \ref{SzwarcLemma1}, we have to check $g_j \in [0,1]$ for all $j \in \N$. Let us first consider $(g_j)_{j \in 2\N-1}$ and its iteration function \eqref{IterationFunction}. As seen in \eqref{FixedPoint} the fixed points of $f$ are given by 
\[x^* = \frac{1 + \eta_2(A) - \eta_1(A) \pm \sqrt{(1 + \eta_2(A) - \eta_1(A))^2 - 4\eta_2(A)}}{2}.\]
Plugging our assumption $\sqrt{\eta_1(A)} + \sqrt{\eta_2(A)} = 1$ into this equation, we get
\begin{align*}
x^* &= \frac{1 + \eta_2(A) - (1 - \sqrt{\eta_2(A)})^2 \pm \sqrt{(1 + \eta_2(A) - (1 - \sqrt{\eta_2(A)})^2)^2 - 4\eta_2(A)}}{2}\\
&= \sqrt{\eta_2(A)} \pm \frac{\sqrt{4\eta_2(A) - 4\eta_2(A)}}{2}\\
&= \sqrt{\eta_2(A)}.
\end{align*}
Thus there is only one fixed point and $x^* < 1$. By \eqref{df/dx}, the iteration function $f$ is strictly increasing in $(0,1 - \eta_1(A))$, while
\[1 - \eta_1(A) = 1 - (1 - \sqrt{\eta_2(A)})^2 = 2\sqrt{\eta_2(A)} - \eta_2(A) > \sqrt{\eta_2(A)}\]
since $\eta_2(A) = x^* < 1$. Furthermore, $g_1 = 0$ and thus $0 \leq g_j \leq x^* < 1$ for all $j \in 2\N-1$. We conclude $g_j \in [0,1]$ for odd $j$. Similarly (exchanging $\eta_1(A)$ and $\eta_2(A)$ and using the starting point $\eta_1(A) < \sqrt{\eta_1(A)} < 1$), we also get $g_j \in [0,1]$ for even $j$. Furthermore, Condition \eqref{BoundCondition} is fulfilled by definition. Thus $(g_j)_{j \in \N}$ meets all the requirements and we can apply Lemma \ref{SzwarcLemma1} to $C$, which implies $r_0(C) = r_0(A) \leq N$. So let us summarize what we have so far. We have 
\begin{itemize}
	\item[(i)] $\sqrt{\eta_1(A)} + \sqrt{\eta_2(A)} \leq 1$ if $N = r_0(A)$ and 
	\item[(ii)] $N \geq r_0(A)$ if $\sqrt{\eta_1(A)} + \sqrt{\eta_2(A)} = 1$.
\end{itemize}
Now let $\sqrt{\eta_1(A)} + \sqrt{\eta_2(A)} = 1$ and assume $r_0(A) < N$. Then
\begin{align*}
\eta_1(A) < \frac{\abs{A_{1,2} + \overline{A_{2,1}}}^2}{4(r_0(A) - \Re \, A_{1,1})(r_0(A) - \Re \, A_{2,2})} =: \tilde{\eta}_1(A),\\
\eta_2(A) < \frac{\abs{A_{2,3} + \overline{A_{3,2}}}^2}{4(r_0(A) - \Re \, A_{2,2})(r_0(A) - \Re \, A_{3,3})} =: \tilde{\eta}_2(A)
\end{align*}
and thus $\sqrt{\tilde{\eta}_1(A)} + \sqrt{\tilde{\eta}_2(A)} > 1$. But this is a contradiction to (i). Thus $\sqrt{\eta_1(A)} + \sqrt{\eta_2(A)} = 1$ implies $r_0(A) = N$.

Conversely, let $N = r_0(A)$ and assume $\sqrt{\eta_1(A)} + \sqrt{\eta_2(A)} < 1$. Then by continuity there exists an $\epsilon > 0$ such that
\[\sqrt{\frac{\abs{A_{1,2} + \overline{A_{2,1}}}^2}{4(N - \epsilon - \Re A_{1,1})(N - \epsilon - \Re A_{2,2})}} + \sqrt{\frac{\abs{A_{2,3} + \overline{A_{3,2}}}^2}{4(N - \epsilon - \Re A_{2,2})(N - \epsilon - \Re A_{3,3})}} = 1.\]
This is a contradiction to (ii) since $N - \epsilon < r_0(A)$. Thus $N = r_0(A)$ implies $\sqrt{\eta_1(A)} + \sqrt{\eta_2(A)} = 1$.
\end{proof}

Although we will not need this in what follows, it is worth noting that, since $A_{3,3} = A_{1,1}$, the equation $\sqrt{\eta_1(A)} + \sqrt{\eta_2(A)} = 1$ can be solved for $r_0(A)$. Clearly, this formula is also valid if $A + A^*$ is diagonal.

\begin{cor} \label{2percor}
Let $\I \in \set{\N,\Z}$ and let $A \in \Lc(\ell^2(\I))$ be tridiagonal and $2$-periodic. Then
\begin{equation} \label{2pereq}
r_0(A) = \frac{1}{2}(a + b + \sqrt{(a-b)^2 + (c+d)^2} \geq \max\set{a,b}
\end{equation}
with equality if and only if $c = d = 0$, where $a = \Re A_{1,1} = \Re A_{3,3}$, $b = \Re A_{2,2}$, $c = \frac{1}{2}\abs{A_{1,2} + \overline{A_{2,1}}}$ and $d = \frac{1}{2}\abs{A_{2,3} + \overline{A_{3,2}}}$.
\end{cor}

\section{The Feinberg-Zee Random Hopping Matrix}

In this section we consider a generalization of the Feinberg-Zee random hopping matrix that was considered in \cite{ChaDa}:
\[A_{\sigma} := \begin{pmatrix} \ddots & \ddots & & & & \\ \ddots & 0 & 1 & & & \\ & c_{-1} & 0 & 1 & & \\ & & c_0 & 0 & 1 & \\ & & & c_1 & 0 & \smash{\ddots} \\ & & & & \ddots & \ddots \end{pmatrix},\]
where $(c_j)_{j \in \Z}$ is a sequence of i.i.d.~random variables taking values in $\set{\pm \sigma}$ and $\sigma \in (0,1]$. The authors of \cite{ChaDa} showed
\[\spec(A_{\sigma}) \subseteq \set{x + iy : \abs{x} + \abs{y} \leq \sqrt{2(1+\sigma^2)}}.\]
In the case $\sigma = 1$ this square is (almost surely) exactly the numerical range of $A_{\sigma}$ as shown in \cite{ChaChoLi2} by an explicit computation. For $\sigma < 1$ the square is tangential to the ellipses in Theorem \ref{ntridiag} and thus a proper superset of the numerical range of $A_{\sigma}$ (see Proposition \ref{nrFZ} for an explicit formula of $N(A_{\sigma})$). We try to further improve this bound obtained in Theorem \ref{ntridiag} by computing the numerical range of $N(A_{\sigma}^2)$. The idea is the following:
\[\spec(A_{\sigma}) = \set{z \in \C : z \in \spec(A_{\sigma})} \subseteq \set{z \in \C : z^2 \in \spec(A_{\sigma}^2)} \subseteq \set{z \in \C : z^2 \in N(A_{\sigma}^2)} =: \sqrt{N(A_{\sigma}^2)}.\]
We thus obtain another upper bound to the spectrum. As we will see in Section \ref{ExplForm}, we indeed have $\sqrt{N(A_{\sigma}^2)} \subset N(A_{\sigma})$, thus improving the upper bound to the spectrum for all $\sigma \in (0,1]$, in particular improving the upper bound of \cite{ChaChoLi2} that was obtained by a massive numerical computation in the case $\sigma = 1$. To compute $N(A_{\sigma}^2)$ we will observe that, although $A_{\sigma}^2$ is not tridiagonal itself, it can be decomposed into tridiagonal matrices and thus the method introduced in Section \ref{method} can be applied. Explicit formulas for $N(A_{\sigma})$, $N(A_{\sigma})^2$ and $N(A_{\sigma}^2)$ are postponed to Section \ref{ExplForm}. To simplify the notation, we fix $\sigma$ here and drop the index.

\subsection{\texorpdfstring{Computation of $N(A^2)$}{}}

We will prove the following theorem at the end of this section. The sets $N(B_1^2)$, $N(B_2^2)$ and $N(B_2^2)$ are filled ellipses/disks and can be computed explicitly (see Proposition \ref{B1B2B3}). Theorem \ref{N(A^2)} thus provides an explicit formula for the (almost sure) numerical range of $A^2$.

\begin{thm} \label{N(A^2)}
Let $\sigma \in (0,1]$, $U_{-1} = \left\{1\right\}$, $U_0 = \left\{0\right\}$, $U_1 = \left\{\pm \sigma\right\}$ and $A \in M(U_{-1},U_0,U_1)$. Then
\[N(A^2) \subseteq \conv\left(N(B_1^2) \cup N(B_2^2) \cup N(B_3^2)\right),\]
where $B_1 \in M_{per,4}(U_{-1},U_0,U_1)$ is the operator with period $(\sigma,\sigma,\sigma,\sigma)$, $B_2 \in M_{per,4}(U_{-1},U_0,U_1)$ is the operator with period $(-\sigma,-\sigma,\sigma,\sigma)$ and $B_3 \in M_{per,4}(U_{-1},U_0,U_1)$ is the operator with period $(-\sigma,-\sigma,-\sigma,-\sigma)$. If $A \in \PsiE(U_{-1},U_0,U_1)$, then equality holds.
\end{thm}

That in the case $A \in \PsiE(U_{-1},U_0,U_1)$ the right-hand side is a subset of the left-hand side is clear by Theorem \ref{mtnr} and the fact that $\sigma^{\op}(B^2) = \sigma^{\op}(B)^2$ (see Proposition \ref{basic}). Moreover, it is sufficient to prove $N(A^2) \subseteq \conv\left(N(B_1^2) \cup N(B_2^2) \cup N(B_3^2)\right)$ for $A \in \PsiE(U_{-1},U_0,U_1)$ by the same reason. To do so, we need to compute $N(B_i^2)$ for $i \in \set{1,2,3}$ first.

\begin{prop} \label{B1B2B3}
Let $B_1$, $B_2$ and $B_3$ be as above. Then
\begin{align*}
r_\phi(B_1^2) &= 2\sigma\cos(\phi) + \sqrt{(1+\sigma^2)^2\cos(\phi)^2 + (1-\sigma^2)^2\sin(\phi)^2},\\
r_\phi(B_2^2) &= 1 + \sigma^2,\\
r_\phi(B_3^2) &= -2\sigma\cos(\phi) + \sqrt{(1+\sigma^2)^2\cos(\phi)^2 + (1-\sigma^2)^2\sin(\phi)^2}
\end{align*}
and the boundaries of $N(B_1^2)$ and $N(B_2^2)$ are given by the following parametrizations:
\begin{align*}
\partial N(B_1^2) : z(t) &= 2\sigma + (1+\sigma^2)\cos(t) + i(1-\sigma^2)\sin(t),\\
\partial N(B_2^2) : z(t) &= (1+\sigma^2)e^{it},\\
\partial N(B_3^2) : z(t) &= -2\sigma + (1+\sigma^2)\cos(t) + i(1-\sigma^2)\sin(t).
\end{align*}
\end{prop}

\begin{proof}
$B_1$ is a Laurent operator with diagonals $(1)_{i \in \Z}$, $(0)_{i \in \Z}$ and $(\sigma)_{i \in \Z}$ and therefore $B_1^2$ is a Laurent operator with diagonals $(1)_{i \in \Z}$, $(0)_{i \in \Z}$, $(2\sigma)_{i \in \Z}$, $(0)_{i \in \Z}$ and $(\sigma^2)_{i \in \Z}$. Therefore the spectrum of $B_1^2$ is given by the ellipse $E := \left\{t \in [0,2\pi) : 2\sigma + (1+\sigma^2)\cos(t) + i(1-\sigma^2)\sin(t)\right\}$ (see e.g. \cite{LiRo} or use Theorem \ref{spcper}). Since Laurent operators are normal, $E$ is equal to the boundary of the numerical range of $B_1^2$. An elementary computation yields $r_\phi(B_1^2) = 2\sigma\cos(\phi) + \sqrt{(1+\sigma^2)^2\cos(\phi)^2 + (1-\sigma^2)^2\sin(t)^2}$. 

$B_2^2$ is a $4$-periodic operator that looks like this:
\begin{equation}
B_2^2 = \begin{pmatrix} 
\ddots & \ddots   & \ddots    & \ddots   & \ddots    &        &         &        &        &        \\ 
       & \sigma^2 & 0         & 0        & 0         & 1      &         &        &        &        \\
       &          & -\sigma^2 & 0        & -2\sigma  & 0      & 1       &        &        &        \\
       &          &           & \sigma^2 & 0         & 0      & 0       & 1      &        &        \\
       &          &           &          & -\sigma^2 & 0      & 2\sigma & 0      & 1      &        \\
       &          &           &          &           & \ddots & \ddots  & \ddots & \ddots & \ddots
\end{pmatrix}.
\end{equation}
It can be decomposed into an even and an odd part as follows. Let 
\[\Xb_e := \set{x \in \Xb : x_{2j+1} = 0 \text{ for all } j \in \Z} \quad \text{and} \quad \Xb_o := \set{x \in \Xb : x_{2j} = 0 \text{ for all } j \in \Z}.\]
Then $B_2^2(\Xb_e) \subset \Xb_e$ and $B_2^2(\Xb_o) \subset \Xb_o$. Thus we can consider $C_2 := A^2|_{\Xb_e}$ and $D_2 := A^2|_{\Xb_o}$ and get $A^2 = C \oplus D$ w.r.t.~this decomposition of $\Xb$, where $C_2$ and $D_2$ are tridiagonal operators given by
\begin{equation*}
C_2 = \begin{pmatrix} 
\ddots & \ddots   & \ddots    &        &         &        \\ 
       & \sigma^2 & 0         & 1      &         &        \\
       &          & \sigma^2  & 0      & 1       &        \\
       &          &           & \ddots & \ddots  & \ddots
\end{pmatrix}, \qquad
D_2 = \begin{pmatrix} 
\ddots & \ddots    & \ddots     &         &         &        \\ 
       & -\sigma^2 & -2\sigma   & 1       &         &        \\
       &           & -\sigma^2  & 2\sigma & 1       &        \\
       &           &            & \ddots  & \ddots  & \ddots
\end{pmatrix}.
\end{equation*}
We see that $C_2$ is a Laurent operator and similarly as before we conclude that the boundary of the numerical range of $C_2$ is given by the ellipse $\left\{t \in [0,2\pi) : (1+\sigma^2)\cos(t) + i(1-\sigma^2)\sin(t)\right\}$.
$D_2$ is a $2$-periodic operator, hence we can apply Proposition \ref{2per}. Let $D_{2,\phi} := e^{i\phi}D_2$, $N := 1+\sigma^2$ and let us exclude the cases $(\sigma,\phi) = (1,0)$ and $(\sigma,\phi) = (1,\pi)$ for the moment so that $D_{2,\phi} + D_{2,\phi}^*$ is not diagonal. In the notation of Proposition \ref{2per} $\eta_1(D_{2,\phi})$ and $\eta_2(D_{2,\phi})$ are given by
\begin{align*}
\eta_1(D_{2,\phi}) &= \frac{\left|e^{i\phi} - \sigma^2 e^{-i\phi}\right|^2}{4(1 + \sigma^2 + 2\sigma\cos(\phi))(1 + \sigma^2 - 2\sigma\cos(\phi))} = \frac{(1 + \sigma^2)^2 - 4\sigma^2\cos(\phi)^2}{4((1 + \sigma^2)^2 - 4\sigma^2\cos(\phi)^2)} = \frac{1}{4},\\
\eta_2(D_{2,\phi}) &= \eta_1(D_{2,\phi}) = \frac{1}{4}.
\end{align*}
Thus $\sqrt{\eta_1(D_{2,\phi})} + \sqrt{\eta_2(D_{2,\phi})} = 1$ and by Proposition \ref{2per}, $r_\phi(D_2) = 1 + \sigma^2$ for all $\phi \in [0,2\pi)$ ($(\sigma,\phi) \notin \left\{(1,0),(1,\pi)\right\}$). In the remaining two cases $\frac{1}{2}(D_{2,\phi} + D_{2,\phi}^*)$ is a diagonal matrix and thus it is easily seen that $r_\phi(D_2) = 2$ holds. Therefore we have $r_\phi(D_2) = 1 + \sigma^2$ for all $\phi \in [0,2\pi)$. Now obviously $N(C_2) \subset N(D_2)$ holds and thus we get $r_\phi(B_2^2) = 1 + \sigma^2$ for all $\phi \in [0,2\pi)$. A parametrization of $\partial N(B_2^2)$ is then of course given by $z(t) = (1 + \sigma^2)e^{it}, t \in [0,2\pi)$.

$B_3$ is the same as $B_1$ just with $\sigma$ replaced by $-\sigma$.
\end{proof}

Next we have to compute
\begin{equation} \label{N}
N(\phi) := \max\limits_{j \in \set{1,2,3}} r_\phi(B_j^2)
\end{equation}
for every $\phi \in [0,2\pi)$.

\begin{prop} \label{ValuesOfN}
Let $B_1$, $B_2$ and $B_3$ be as above, $\phi^* := \arccos(\frac{\sigma}{1 + \sigma^2})$ and let $N$ be given by \eqref{N}. Then $N$ takes the following values:
\[N(\phi) = \begin{cases} 2\sigma\cos(\phi) + \sqrt{(1 + \sigma^2)^2\cos(\phi)^2 + (1 - \sigma^2)^2\sin(\phi)^2} & \text{if } 0 \leq \phi \leq \phi^*,\\ 1 + \sigma^2 & \text{if } \phi^* \leq \phi \leq \pi - \phi^*,\\ -2\sigma\cos(\phi) + \sqrt{(1 + \sigma^2)^2\cos(\phi)^2 + (1 - \sigma^2)^2\sin(\phi)^2} & \text{if } \pi - \phi^* \leq \phi \leq \pi + \phi^*,\\ 1 + \sigma^2 & \text{if } \pi + \phi^* \leq \phi \leq 2\pi - \phi^*,\\ 2\sigma\cos(\phi) + \sqrt{(1 + \sigma^2)^2\cos(\phi)^2 + (1 - \sigma^2)^2\sin(\phi)^2} & \text{if } 2\pi - \phi^* \leq \phi \leq 2\pi.\end{cases}\]
\end{prop}

\begin{proof}
Since all of these functions are continuous, we only have to check where the graphs of $r_\phi(B_1^2)$, $r_\phi(B_2^2)$ and $r_\phi(B_3^2)$ intersect. Let us have a look at $r_\phi(B_1^2)$ and $r_\phi(B_2^2)$ first:
\begin{align*}
r_\phi(B_1^2) = r_\phi(B_2^2) \quad &\Leftrightarrow \quad 2\sigma\cos(\phi) + \sqrt{(1 + \sigma^2)^2\cos(\phi)^2 + (1 - \sigma^2)^2\sin(\phi)^2} = 1 + \sigma^2\\
&\Leftrightarrow \quad (1 + \sigma^2)^2\cos(\phi)^2 + (1 - \sigma^2)^2(1 - \cos(\phi)^2) = (1 + \sigma^2 - 2\sigma\cos(\phi))^2\\
&\Leftrightarrow \quad \cos(\phi) = \frac{\sigma}{1 + \sigma^2}.
\end{align*}
Thus the graphs of $r_\phi(B_1^2)$ and $r_\phi(B_2^2)$ only intersect at $\phi^* = \arccos(\frac{\sigma}{1 + \sigma^2})$ and $2\pi - \phi^*$. Similarly, the graphs of $r_\phi(B_2^2)$ and $r_\phi(B_3^2)$ only intersect at $\pi - \phi^* = \arccos(\frac{-\sigma}{1 + \sigma^2})$ and $\pi + \phi^*$. Finally, $r_\phi(B_1^2)$ and $r_\phi(B_3^2)$ obviously only intersect at $\frac{\pi}{2}$ and $\frac{3\pi}{2}$. Plugging in some angles and using \eqref{N}, one easily deduces the assertion.
\end{proof}

Now let us focus on $A^2$. Let us denote the first subdiagonal of $A \in \PsiE(U_{-1},U_0,U_1)$ by $(h_j)_{j \in \Z}$, i.e.~$h_j := A_{j+1,j}$ for all $j \in \Z$. Then $A^2$ has the following entries:
\begin{align*}
(A^2)_{j,j+2} &= A_{j,j+1}A_{j+1,j+2} = 1,\\
(A^2)_{j,j+1} &= A_{j,j+1}A_{j+1,j+1} + A_{j,j}A_{j,j+1} = 0,\\
(A^2)_{j,j} &= A_{j,j+1}A_{j+1,j} + A_{j,j}A_{j,j} + A_{j,j-1}A_{j-1,j} = h_j + h_{j-1},\\
(A^2)_{j,j-1} &= A_{j,j}A_{j,j-1} + A_{j,j-1}A_{j-1,j-1} = 0,\\
(A^2)_{j,j-2} &= A_{j,j-1}A_{j-1,j-2} = h_{j-1}h_{j-2}
\end{align*}
and can be decomposed as $A^2 = C \oplus D$ as in the proof of Proposition \ref{B1B2B3}. The matrices $C$ and $D$ are given by
\begin{align*}
C_{j,j+1} &= 1,\\
C_{j,j} &= h_{2j} + h_{2j-1},\\
C_{j,j-1} &= h_{2j-1}h_{2j-2}
\end{align*}
and
\begin{align*}
D_{j,j+1} &= 1,\\
D_{j,j} &= h_{2j+1} + h_{2j},\\
D_{j,j-1} &= h_{2j}h_{2j-1}
\end{align*}
for $j \in \Z$, respectively. We will focus on the computation of the numerical range of $C$. The computation of the numerical range of $D$ is exactly the same so that we obtain $N(C) = N(D)$. Since the numerical range of a direct sum is just the convex hull of the union of the numerical ranges of its components, we get $N(A^2) = N(C) = N(D)$.

By Proposition \ref{lope}, we have $A \in \sigma^{\op}(A)$ and thus there exists a sequence of integers $(g_n)_{n \in \N}$ tending to infinity such that $A_g$ exists and is equal to $A$. W.l.o.g.~we may assume that this sequence tends to $+\infty$. Then $(A^2)_g = (A_g)^2 = A^2 = C \oplus D$ by Proposition \ref{basic}. Observe that $V_{-g_n}(C \oplus D)V_{g_n} = V_{-g_n/2}CV_{g_n/2} \oplus V_{-g_n/2}DV_{g_n/2}$ if $g_n$ is even and $V_{-g_n}(C \oplus D)V_{g_n} = V_{-(g_n-1)/2}DV_{(g_n-1)/2} \oplus V_{-(g_n+1)/2}CV_{(g_n+1)/2}$ if $g_n$ is odd. Clearly either $\set{n \in \N: g_n \text{ is even}}$ or $\set{n \in \N : g_n \text{ is odd}}$ is an infinite set. Let us first assume that $\set{n \in \N: g_n \text{ is even}}$ is infinite and denote the sequence of even elements in $g$ by $g^e$. Then by construction $V_{-g^e_n/2}CV_{g^e_n/2}$ converges strongly to $C$ and $V_{-g^e_n/2}DV_{g^e_n/2}$ converges strongly to $D$ as $n \to \infty$. Thus $C \in \sigma^{\op}(C)$ and $D \in \sigma^{\op}(D)$. Similarly, assume that $\set{n \in \N: g_n \text{ is odd}}$ is infinite and denote the sequence of odd elements in $g$ by $g^o$. Then by construction, $V_{-(g^o_n-1)/2}CV_{(g^o_n-1)/2}$ converges strongly to $D$ and $V_{-(g^o_n+1)/2}DV_{(g^o_n+1)/2}$ converges strongly to $C$ as $n \to \infty$. Thus $D \in \sigma^{\op}(C)$ and $C \in \sigma^{\op}(D)$ in this case. Since limit operators of limit operators are again limit operators of the original operator (see e.g.~\cite[Corollary 3.97]{Marko}), we also get $C \in \sigma^{\op}(C)$ and $D \in \sigma^{\op}(D)$ in this case. Since $g^e$ and $g^o$ tend to $+\infty$, we can apply Proposition \ref{selfsim} to get $N(A^2) = N(C) = N(C_+)$, where $C_+ := P_{\N}CP_{\N}|_{\im P_{\N}} \in \Lc(\ell^2(\N))$.

Fix $\phi \in [0,2\pi)$ and let $E(\phi)$ be the real symmetric tridiagonal operator that satisfies
\begin{align*}
E_{j,j}(\phi) &= \Re(e^{i\phi}(C_+)_{j,j}),\\
E_{j,j+1}(\phi) &= \frac{1}{2}\abs{e^{i\phi}(C_+)_{j,j+1} + e^{-i\phi}\overline{(C_+)_{j+1,j}}}
\end{align*}
and $r_\phi(A^2) = r_{\phi}(C_+) = r_0(E(\phi))$ (cf. \eqref{C}). Now for every angle $\phi$ there are 16 different combinations for $(h_{2j-1}, h_{2j}, h_{2j+1}, h_{2j+2})$ in \eqref{BoundCondition}. Define
\begin{equation} \label{FormulaEta}
\eta_j(\phi) := \frac{E_{j,j+1}(\phi)^2}{(N(\phi) - E_{j,j}(\phi))(N(\phi) - E_{j+1,j+1}(\phi))}
\end{equation}
for all $j \in \N$, where $N(\phi)$ is given by Proposition \ref{ValuesOfN}. Let us consider $\phi \in [\phi^*,\frac{\pi}{2}]$ first. For these angles, we have the following table. For later reference we numbered the $16$ cases lexicographically.

\renewcommand{\arraystretch}{1.5}

\begin{center}
\large{
\begin{tabular}{|c|c|c|} \hline
$t_j$ & $(h_{2j-1}, h_{2j}, h_{2j+1}, h_{2j+2})$ & $\eta_j(\phi)$\\ \hline
$1$ & $(\sigma,\sigma,\sigma,\sigma)$ & $\frac{(1 - \sigma^2)^2 + 4\sigma^2\cos(\phi)^2}{4(1 + \sigma^2 - 2\sigma\cos(\phi))^2}$\\ \hline
$2$ & $(\sigma,\sigma,\sigma,-\sigma)$ & $\frac{(1 - \sigma^2)^2 + 4\sigma^2\cos(\phi)^2}{4(1 + \sigma^2 - 2\sigma\cos(\phi))(1 + \sigma^2)}$\\ \hline
$3$ & $(\sigma,\sigma,-\sigma,\sigma)$ & $\frac{(1 + \sigma^2)^2 - 4\sigma^2\cos(\phi)^2}{4(1 + \sigma^2 - 2\sigma\cos(\phi))(1 + \sigma^2)}$\\ \hline
$4$ & $(\sigma,\sigma,-\sigma,-\sigma)$ & $\frac{(1 + \sigma^2)^2 - 4\sigma^2\cos(\phi)^2}{4(1 + \sigma^2 - 2\sigma\cos(\phi))(1 + \sigma^2 + 2\sigma\cos(\phi))}$\\ \hline
$5$ & $(\sigma,-\sigma,\sigma,\sigma)$ & $\frac{(1 + \sigma^2)^2 - 4\sigma^2\cos(\phi)^2}{4(1 + \sigma^2 - 2\sigma\cos(\phi))(1 + \sigma^2)}$\\ \hline
$6$ & $(\sigma,-\sigma,\sigma,-\sigma)$ & $\frac{(1 + \sigma^2)^2 - 4\sigma^2\cos(\phi)^2}{4(1 + \sigma^2)^2}$\\ \hline
$7$ & $(\sigma,-\sigma,-\sigma,\sigma)$ & $\frac{(1 - \sigma^2)^2 + 4\sigma^2\cos(\phi)^2}{4(1 + \sigma^2)^2}$\\ \hline
$8$ & $(\sigma,-\sigma,-\sigma,-\sigma)$ & $\frac{(1 - \sigma^2)^2 + 4\sigma^2\cos(\phi)^2}{4(1 + \sigma^2 + 2\sigma\cos(\phi))(1 + \sigma^2)}$\\ \hline
$9$ & $(-\sigma,\sigma,\sigma,\sigma)$ & $\frac{(1 - \sigma^2)^2 + 4\sigma^2\cos(\phi)^2}{4(1 + \sigma^2 - 2\sigma\cos(\phi))(1 + \sigma^2)}$\\ \hline
$10$ & $(-\sigma,\sigma,\sigma,-\sigma)$ & $\frac{(1 - \sigma^2)^2 + 4\sigma^2\cos(\phi)^2}{4(1 + \sigma^2)^2}$\\ \hline
$11$ & $(-\sigma,\sigma,-\sigma,\sigma)$ & $\frac{(1 + \sigma^2)^2 - 4\sigma^2\cos(\phi)^2}{4(1 + \sigma^2)^2}$\\ \hline
$12$ & $(-\sigma,\sigma,-\sigma,-\sigma)$ & $\frac{(1 + \sigma^2)^2 - 4\sigma^2\cos(\phi)^2}{4(1 + \sigma^2 + 2\sigma\cos(\phi))(1 + \sigma^2)}$\\ \hline
$13$ & $(-\sigma,-\sigma,\sigma,\sigma)$ & $\frac{(1 + \sigma^2)^2 - 4\sigma^2\cos(\phi)^2}{4(1 + \sigma^2 - 2\sigma\cos(\phi))(1 + \sigma^2 + 2\sigma\cos(\phi))}$\\ \hline
$14$ & $(-\sigma,-\sigma,\sigma,-\sigma)$ & $\frac{(1 + \sigma^2)^2 - 4\sigma^2\cos(\phi)^2}{4(1 + \sigma^2 + 2\sigma\cos(\phi))(1 + \sigma^2)}$\\ \hline
$15$ & $(-\sigma,-\sigma,-\sigma,\sigma)$ & $\frac{(1 - \sigma^2)^2 + 4\sigma^2\cos(\phi)^2}{4(1 + \sigma^2 + 2\sigma\cos(\phi))(1 + \sigma^2)}$\\ \hline
$16$ & $(-\sigma,-\sigma,-\sigma,-\sigma)$ & $\frac{(1 - \sigma^2)^2 + 4\sigma^2\cos(\phi)^2}{4(1 + \sigma^2 + 2\sigma\cos(\phi))^2}$\\ \hline
\end{tabular}
}
\captionof{table}{} \label{table1}
\end{center}

This table has to be read as follows. The sequence $(h_j)_{j \in \N}$ induces a sequence $(t_j)_{j \in \N}$. For example if the sequence $(h_j)_{j \in \N}$ starts with $(\sigma,-\sigma,-\sigma,\sigma,\sigma,\sigma,\sigma,-\sigma,\sigma,-\sigma,\ldots)$, the sequence $(t_j)_{j \in \N}$ starts with $(7,9,2,6,\ldots)$. The numbers $t_j$ are used to refer to the respective $\eta_j$, which are computed via Formula \eqref{FormulaEta}. So if, for example, $t_j = 6$, then $\eta_j(\phi) = \frac{(1 + \sigma^2)^2 - 4\sigma^2\cos(\phi)^2}{4(1 + \sigma^2)^2}$.

We will find the following equalities and inequalities useful:
\begin{align}
0 &\leq \cos(\phi) \leq \frac{\sigma}{1 + \sigma^2} < 1\label{cosineq1}\\
(1 - \sigma^2)^2 + 4\sigma^2\cos(\phi)^2 &\leq (1 - \sigma^2)^2 + \frac{4\sigma^4}{(1 + \sigma^2)^2} = \frac{(1 + \sigma^4)^2}{(1 + \sigma^2)^2}\label{cosineq2}\\
1 + \sigma^2 - 2\sigma\cos(\phi) &\geq 1 + \sigma^2 - \frac{2\sigma^2}{1 + \sigma^2} = \frac{1 + \sigma^4}{1 + \sigma^2}\label{cosineq3}\\
(1 + \sigma^2)^2 - 4\sigma^2\cos(\phi)^2 &= (1 + \sigma + 2\sigma\cos(\phi))(1 + \sigma - 2\sigma\cos(\phi))\label{cosineq4}
\end{align}

Using these, it is not difficult to see that $\eta_j(\phi) \leq \frac{1}{2}$ for all $\phi \in [\phi^*,\frac{\pi}{2}]$ and $j \in \N$ (i.e.~for all possible values of $\eta_j(\phi)$ in Table \ref{table1}). We even have $\eta_j(\phi) \leq \frac{1}{4}$ for all $\phi \in [\phi^*,\frac{\pi}{2}]$ and $j \in \N$ with $t_j \notin \set{3,5}$. This observation is very useful to finally construct the sequence needed for Lemma \ref{SzwarcLemma1}.

\begin{prop} \label{case1}
Let $\sigma \in (0,1]$, $U_{-1} = \left\{1\right\}$, $U_0 = \left\{0\right\}$, $U_1 = \left\{\pm \sigma\right\}$ and let $A \in \Psi E(U_{-1},U_0,U_1)$. Let $\phi \in [\phi^*,\frac{\pi}{2}]$, $\eta_j := \eta_j(\phi)$ and $t_j$ for all $j \in \N$ be defined as above. Then the sequence $(g_j)_{j \in \N}$, defined  by the following prescription, satisfies $g_j \in [0,1]$ and $\eta_j \leq g_{j+1}(1 - g_j)$ for all $j \in \N$:
\begin{itemize}
	\item If $t_1 = 5$, choose $g_1 = \frac{1}{2}\frac{1 + \sigma^2 - 2\sigma\cos(\phi)}{1 + \sigma^2}$.
	\item If there is some $k \in \N$ such that $t_1 = \ldots = t_k = 6$ and $t_{k+1} = 5$, choose $g_1 = \frac{1}{2}\frac{1 + \sigma^2 - 2\sigma\cos(\phi)}{1 + \sigma^2}$.
	\item If neither is true, choose $g_1 = \frac{1}{2}$.
	\item If $t_j \in \set{2,6,10,14}$ and $t_{j+1} = 5$, choose $g_{j+1} = \frac{1}{2}\frac{1 + \sigma^2 - 2\sigma\cos(\phi)}{1 + \sigma^2}$.
	\item If $t_j \in \set{2,6,10,14}$, there is some $k > j$ such that $t_{j+1} = \ldots = t_k = 6$ and $t_{k+1} = 5$, choose $g_{j+1} = \frac{1}{2}\frac{1 + \sigma^2 - 2\sigma\cos(\phi)}{1 + \sigma^2}$.
	\item If $t_j = 3$, choose $g_{j+1} = \frac{1}{2}\frac{1 + \sigma^2 + 2\sigma\cos(\phi)}{1 + \sigma^2}$.
	\item If $t_j = 11$, thereis some $k \leq j$ such that $t_k = \ldots = t_j = 11$ and $t_{k-1} = 3$, choose $g_{j+1} = \frac{1}{2}\frac{1 + \sigma^2 + 2\sigma\cos(\phi)}{1 + \sigma^2}$.
	\item If none of the above is true, choose $g_{j+1} = \frac{1}{2}$.
\end{itemize}
\end{prop}

\begin{proof}
That $g_j \in [0,1]$ holds for all $j \in \N$ follows from \eqref{cosineq1}. So it remains to prove that $\eta_j \leq g_{j+1}(1 - g_j)$ holds. Above we observed that $\eta_j \leq \frac{1}{4}$ unless $t_j \in \set{3,5}$. So if $t_j \notin \set{3,5}$ for all $j \in \N$, then $\eta_j \leq g_{j+1}(1-g_j)$ is obviously satisfied. It remains to investigate what happens if $t_j \in \set{3,5}$ for some $j \in \N$. Roughly speaking, the idea is that the cases $t_j = 3$ and $t_j = 5$ affect the sequence $(g_k)_{k \in \N}$ only locally in the sense that $\set{k \in \N : g_k = \frac{1}{2}}$ is an infinite set. Thus if $t_j \in \set{3,5}$ occurs, we try to get back to $\frac{1}{2}$ as soon as possible as $j$ increases. The argument can then be repeated by induction.

Note that if $t_j \in \set{3,5}$, we can simplify $\eta_j$ as follows:
\[\eta_j = \frac{(1 + \sigma^2)^2 - 4\sigma^2\cos(\phi)^2}{4(1 + \sigma^2 - 2\sigma\cos(\phi))(1 + \sigma^2)} = \frac{1}{4}\frac{1 + \sigma^2 + 2\sigma\cos(\phi)}{1 + \sigma^2},\]
where we used \eqref{cosineq4}.

Let us consider the case $t_j = 3$ first and assume $g_j = \frac{1}{2}$. More precisely, we start our sequence with $g_1 = g_2 = \ldots = \frac{1}{2}$ until $t_j \in \set{3,5}$ occurs the first time and consider the case where $t_j = 3$ occurs first. Then by definition $g_{j+1} = \frac{1}{2}\frac{1 + \sigma^2 + 2\sigma\cos(\phi)}{1 + \sigma^2}$ and
\[g_{j+1}(1 - g_j) = \frac{1}{4}\frac{1 + \sigma^2 + 2\sigma\cos(\phi)}{1 + \sigma^2} = \eta_j.\]

Observe that $\eta_j$ and $\eta_{j+1}$ are not independent. Indeed, $\eta_{j+1}$ depends on $h_{2j+1}$, $h_{2j+2}$, $h_{2j+3}$ and $h_{2j+4}$ whereas $\eta_j$ depends on $h_{2j-1}$, $h_{2j}$, $h_{2j+1}$ and $h_{2j+2}$. Thus if we fix $\eta_j$, there are only 4 possible combinations for $\eta_{j+1}$. In particular, if $t_j = 3$, then $t_{j+1}$ has to be contained in $\set{9,10,11,12}$. So there are four cases:
\begin{align*}
\eta_{j+1} &= \frac{(1 - \sigma^2)^2 + 4\sigma^2\cos(\phi)^2}{4(1 + \sigma^2 - 2\sigma\cos(\phi))(1 + \sigma^2)} \quad (t_{j+1} = 9),\\
\eta_{j+1} &= \frac{(1 - \sigma^2)^2 + 4\sigma^2\cos(\phi)^2}{4(1 + \sigma^2)^2} \quad (t_{j+1} = 10),\\
\eta_{j+1} &= \frac{(1 + \sigma^2)^2 - 4\sigma^2\cos(\phi)^2}{4(1 + \sigma^2)^2} \quad (t_{j+1} = 11),\\
\eta_{j+1} &= \frac{(1 + \sigma^2)^2 - 4\sigma^2\cos(\phi)^2}{4(1 + \sigma^2 + 2\sigma\cos(\phi))(1 + \sigma^2)} \quad (t_{j+1} = 12).
\end{align*}
In the first case we have $g_{j+2} = \frac{1}{2}$:
\begin{align*}
g_{j+2}(1 - g_{j+1}) &= \frac{1}{2} - \frac{1}{4}\frac{1 + \sigma^2 + 2\sigma\cos(\phi)}{1 + \sigma^2}\\
&= \frac{1}{4}\frac{1 + \sigma^2 - 2\sigma\cos(\phi)}{1 + \sigma^2} \\
&\geq \frac{1}{4}\frac{1 + \sigma^4}{(1 + \sigma^2)^2}\\
&\geq \eta_{j+1},
\end{align*}
where we used \eqref{cosineq3} in line 2 and \eqref{cosineq2} and \eqref{cosineq3} in line 3. In the second case we have $g_{j+2} = \frac{1}{2}\frac{1 + \sigma^2 - 2\sigma\cos(\phi)}{1 + \sigma^2} \leq \frac{1}{2}$ if $t_{j+2} \in \set{5,6}$ and $g_{j+2} = \frac{1}{2}$ if not:
\begin{align*}
g_{j+2}(1 - g_{j+1}) &\geq \frac{1 + \sigma^2 - 2\sigma\cos(\phi)}{1 + \sigma^2}\left(\frac{1}{2} - \frac{1}{4}\frac{1 + \sigma^2 + 2\sigma\cos(\phi)}{1 + \sigma^2}\right)\\
&= \frac{1}{4}\frac{(1 + \sigma^2 - 2\sigma\cos(\phi))^2}{(1 + \sigma^2)^2}\\
&\geq \frac{1}{4}\frac{(1 + \sigma^4)^2}{(1 + \sigma^2)^4}\\
&\geq \eta_{j+1},
\end{align*}
where we used \eqref{cosineq3} in line 2 and \eqref{cosineq2} in line 3. In the third case we have $g_{j+2} = \frac{1}{2}\frac{1 + \sigma^2 + 2\sigma\cos(\phi)}{1 + \sigma^2}$:
\begin{align*}
g_{j+2}(1 - g_{j+1}) &= \frac{1 + \sigma^2 + 2\sigma\cos(\phi)}{1 + \sigma^2}\left(\frac{1}{2} - \frac{1}{4}\frac{1 + \sigma^2 + 2\sigma\cos(\phi)}{1 + \sigma^2}\right)\\
&= \frac{1}{4}\frac{1 + \sigma^2 + 2\sigma\cos(\phi)}{1 + \sigma^2}\frac{1 + \sigma^2 - 2\sigma\cos(\phi)}{1 + \sigma^2}\\
&= \frac{1}{4}\frac{(1 + \sigma^2)^2 - 4\sigma^2\cos(\phi)^2}{(1 + \sigma^2)^2}\\
&= \eta_{j+1}.
\end{align*}
In the fourth case we have $g_{j+2} = \frac{1}{2}$:
\begin{align*}
g_{j+2}(1 - g_{j+1}) &= \frac{1}{2} - \frac{1}{4}\frac{1 + \sigma^2 + 2\sigma\cos(\phi)}{1 + \sigma^2}\\
&= \frac{1}{4}\frac{1 + \sigma^2 - 2\sigma\cos(\phi)}{1 + \sigma^2}\\
&= \frac{1}{4}\frac{(1 + \sigma^2)^2 - 4\sigma^2\cos(\phi)^2}{(1 + \sigma^2 + 2\sigma\cos(\phi))(1 + \sigma^2)}\\
&= \eta_{j+1}.
\end{align*}
So either $g_{j+2} \leq \frac{1}{2}$ (and we included one special case that we need afterwards) or $g_{j+2} = g_{j+1}$. Thus either we are where we started with, namely $\frac{1}{2}$, or we are in the third case, where $\eta_{j+1}$ is of type (11). But in this case we have $h_{2j+1} = h_{2j+3}$ and $h_{2j+2} = h_{2j+4}$ and thus we have again the same four cases for $\eta_{j+2}$ and so on. So either we end up with an infinite sequence with $g_k = g_{j+1}$ for all $k > j$ (which is impossible by pseudo-ergodicity, but would still be just fine) or we eventually go out with $g_k \leq \frac{1}{2}$ for some $k \geq j+2$. Thus we are done by induction if we can control the case $t_j = 5$ as well.

The case $t_j = 5$ is very similar to the case $t_j = 3$, but we have to think backwards this time, which is a little bit more complicated. If we have a look at the generators (i.e.~$h_{2j-1}$, $h_{2j}$, $h_{2j+1}$ and $h_{2j+2}$) of the cases $t_j = 3$ and $t_j = 5$, it is intuitively clear, why this has to be the same but backwards. So assume $t_j = 5$. Then $g_j = \frac{1}{2}\frac{1 + \sigma^2 - 2\sigma\cos(\phi)}{1 + \sigma^2}$ and $g_{j+1} = \frac{1}{2}$ by definition and thus
\[g_{j+1}(1 - g_j) = \frac{1}{2} - \frac{1}{4}\frac{1 + \sigma^2 - 2\sigma\cos(\phi)}{1 + \sigma^2} = \frac{1}{4}\frac{1 + \sigma^2 + 2\sigma\cos(\phi)}{1 + \sigma^2} = \eta_j.\]
As already mentioned, we have to look backwards here, i.e. we want to control $g_{j-1}$. Now there are five cases. The first case is $j = 1$, which is trivial of course. The second case is where $t_{j-1} = 2$. In this case we have $g_{j-1} = \frac{1}{2}$:
\begin{align*}
g_j(1 - g_{j-1}) &= \frac{1}{4}\frac{1 + \sigma^2 - 2\sigma\cos(\phi)}{1 + \sigma^2}\\
&\geq \frac{1}{4} \frac{1 + \sigma^4}{(1 + \sigma^2)^2}\\
&\geq \eta_{j-1},
\end{align*}
where we used \eqref{cosineq3} in line 1 and \eqref{cosineq2} and \eqref{cosineq3} in line 2. The third case is where $t_{j-1} = 6$. In this case we have $g_{j-1} = \frac{1}{2}\frac{1 + \sigma^2 - 2\sigma\cos(\phi)}{1 + \sigma^2}$:
\begin{align*}
g_j(1 - g_{j-1}) &= \frac{1}{2}\frac{1 + \sigma^2 - 2\sigma\cos(\phi)}{1 + \sigma^2}\left(1 - \frac{1}{2}\frac{1 + \sigma^2 - 2\sigma\cos(\phi)}{1 + \sigma^2}\right)\\
&= \frac{1}{4}\frac{1 + \sigma^2 - 2\sigma\cos(\phi)}{1 + \sigma^2}\frac{1 + \sigma^2 + 2\sigma\cos(\phi)}{1 + \sigma^2}\\
&= \frac{1}{4}\frac{(1 + \sigma^2)^2 - 4\sigma^2\cos(\phi)^2}{(1 + \sigma^2)^2}\\
&= \eta_{j-1}.
\end{align*}
The fourth case is where $t_{j-1} = 10$. In this case we either have $g_{j-1} = \frac{1}{2}\frac{1 + \sigma^2 + 2\sigma\cos(\phi)}{1 + \sigma^2} \geq \frac{1}{2}$ if $t_{j-2} \in \set{3,11}$ or $g_{j-1} = \frac{1}{2}$ if not:
\begin{align*}
g_j(1 - g_{j-1}) &\geq \frac{1}{2}\frac{1 + \sigma^2 - 2\sigma\cos(\phi)}{1 + \sigma^2}\left(1 - \frac{1}{2}\frac{1 + \sigma^2 + 2\sigma\cos(\phi)}{1 + \sigma^2}\right)\\
&= \frac{1}{4}\frac{(1 + \sigma^2 - 2\sigma\cos(\phi))^2}{(1 + \sigma^2)^2}\\
&\geq \frac{1}{4}\frac{(1 + \sigma^4)^2}{(1 + \sigma)^4}\\
&\geq \eta_{j-1},
\end{align*}
where we used \eqref{cosineq3} in line 2 and \eqref{cosineq2} and in line 3. Note that this case matches perfectly with the second case above. The fifth case is where $t_{j-1} = 14$. In this case we have $g_{j-1} = \frac{1}{2}$:
\begin{align*}
g_j(1 - g_{j-1}) &= \frac{1}{4}\frac{1 + \sigma^2 - 2\sigma\cos(\phi)}{1 + \sigma^2}\\
&= \frac{1}{4}\frac{(1 + \sigma^2)^2 - 4\sigma^2\cos(\phi)^2}{(1 + \sigma^2 + 2\sigma\cos(\phi))(1 + \sigma^2)}\\
&= \eta_{j-1}.
\end{align*}
Again we conclude that either $g_{j-1} \geq \frac{1}{2}$ (note that the inequality is in the other direction this time, which is good!) or $g_{j-1} = g_j$. Thus either we started where we ended, namely $\frac{1}{2}$ (or even better, we started with something that is greater than or equal to $\frac{1}{2}$ and the sequence reduced to $\frac{1}{2}$, compare with the mentioned special case above), or we are in the third case, where $t_{j-1} = 6$. But in this case we have $h_{2j-1} = h_{2j-3}$ and $h_{2j-2} = h_{2j-4}$ and thus we again have the same four cases for $\eta_{j-2}$ and so on. Thus we either end up at $g_1$, which is fine or we eventually have $g_k \geq \frac{1}{2}$ for some $k \leq j-1$. In either case we are done by induction.
\end{proof}

So we are done with the case $\phi \in [\phi^*,\frac{\pi}{2}]$. This means that there is only the case $\phi \in [0,\phi^*]$ left. All the other angles will follow by symmetry. Let us now consider the table for the angles $\phi \in [0,\phi^*]$. Remember that we have
\[N(\phi) = 2\sigma\cos(\phi) + \sqrt{( 1 + \sigma^2)^2\cos(\phi)^2 + (1 - \sigma^2)^2\sin(\phi)^2} = 2\sigma\cos(\phi) + \sqrt{(1 - \sigma^2)^2 + 4\sigma^2\cos(\phi)^2}\]
here and let us drop the $\phi$ in $N(\phi)$ for the sake of readability.

\renewcommand{\arraystretch}{1.5}

\begin{center}
\large{
\begin{tabular}{|c|c|c|} \hline
$t_j$ & $(h_{2j-1}, h_{2j}, h_{2j+1}, h_{2j+2})$ & $\eta_j(\phi)$\\ \hline
$1$ & $(\sigma,\sigma,\sigma,\sigma)$ & $\frac{1}{4}$\\ \hline
$2$ & $(\sigma,\sigma,\sigma,-\sigma)$ & $\frac{(1 - \sigma^2)^2 + 4\sigma^2\cos(\phi)^2}{4(N - 2\sigma\cos(\phi))N}$\\ \hline
$3$ & $(\sigma,\sigma,-\sigma,\sigma)$ & $\frac{(1 + \sigma^2)^2 - 4\sigma^2\cos(\phi)^2}{4(N - 2\sigma\cos(\phi))N}$\\ \hline
$4$ & $(\sigma,\sigma,-\sigma,-\sigma)$ & $\frac{(1 + \sigma^2)^2 - 4\sigma^2\cos(\phi)^2}{4(N - 2\sigma\cos(\phi))(N + 2\sigma\cos(\phi))}$\\ \hline
$5$ & $(\sigma,-\sigma,\sigma,\sigma)$ & $\frac{(1 + \sigma^2)^2 - 4\sigma^2\cos(\phi)^2}{4(N - 2\sigma\cos(\phi))N}$\\ \hline
$6$ & $(\sigma,-\sigma,\sigma,-\sigma)$ & $\frac{(1 + \sigma^2)^2 - 4\sigma^2\cos(\phi)^2}{4N^2}$\\ \hline
$7$ & $(\sigma,-\sigma,-\sigma,\sigma)$ & $\frac{(1 - \sigma^2)^2 + 4\sigma^2\cos(\phi)^2}{4N^2}$\\ \hline
$8$ & $(\sigma,-\sigma,-\sigma,-\sigma)$ & $\frac{(1 - \sigma^2)^2 + 4\sigma^2\cos(\phi)^2}{4(N + 2\sigma\cos(\phi))N}$\\ \hline
$9$ & $(-\sigma,\sigma,\sigma,\sigma)$ & $\frac{(1 - \sigma^2)^2 + 4\sigma^2\cos(\phi)^2}{4(N - 2\sigma\cos(\phi))N}$\\ \hline
$10$ & $(-\sigma,\sigma,\sigma,-\sigma)$ & $\frac{(1 - \sigma^2)^2 + 4\sigma^2\cos(\phi)^2}{4N^2}$\\ \hline
$11$ & $(-\sigma,\sigma,-\sigma,\sigma)$ & $\frac{(1 + \sigma^2)^2 - 4\sigma^2\cos(\phi)^2}{4N^2}$\\ \hline
$12$ & $(-\sigma,\sigma,-\sigma,-\sigma)$ & $\frac{(1 + \sigma^2)^2 - 4\sigma^2\cos(\phi)^2}{4(N+ 2\sigma\cos(\phi))N}$\\ \hline
$13$ & $(-\sigma,-\sigma,\sigma,\sigma)$ & $\frac{(1 + \sigma^2)^2 - 4\sigma^2\cos(\phi)^2}{4(N - 2\sigma\cos(\phi))(N + 2\sigma\cos(\phi))}$\\ \hline
$14$ & $(-\sigma,-\sigma,\sigma,-\sigma)$ & $\frac{(1 + \sigma^2)^2 - 4\sigma^2\cos(\phi)^2}{4(N + 2\sigma\cos(\phi))N}$\\ \hline
$15$ & $(-\sigma,-\sigma,-\sigma,\sigma)$ & $\frac{(1 - \sigma^2)^2 + 4\sigma^2\cos(\phi)^2}{4(N + 2\sigma\cos(\phi))N}$\\ \hline
$16$ & $(-\sigma,-\sigma,-\sigma,-\sigma)$ & $\frac{(1 - \sigma^2)^2 + 4\sigma^2\cos(\phi)^2}{4(N + 2\sigma\cos(\phi))^2}$\\ \hline
\end{tabular}
}
\captionof{table}{} \label{table2}
\end{center}

We will find the following equalities and inequalities useful:
\begin{align}
N &\geq 1 + \sigma^2\label{cosineq5}\\
\cos(\phi) &\geq \frac{\sigma}{1 + \sigma^2}\label{cosineq6}\\
N - 2\sigma\cos(\phi) &= \sqrt{(1 - \sigma^2)^2 + 4\sigma^2\cos(\phi)^2} \geq \sqrt{(1 - \sigma^2)^2 +\frac{4\sigma^4}{(1 + \sigma^2)^2}} = \frac{1 + \sigma^4}{1 + \sigma^2}\label{cosineq7}\\
N + 2\sigma\cos(\phi) &\geq \frac{1 + \sigma^4}{1 + \sigma^2} + \frac{4\sigma^2}{1 + \sigma^2} = \frac{1 + 4\sigma^2 + \sigma^4}{1 + \sigma^2}\label{cosineq8}\\
(1 - \sigma^2)^2 + 4\sigma^2\cos(\phi)^2 &= (N - 2\sigma\cos(\phi))^2\label{cosineq9}\\
(1 + \sigma^2)^2 - 4\sigma^2\cos(\phi)^2 &\leq (1 + \sigma^2)^2 - \frac{4\sigma^4}{(1 + \sigma^2)^2} = \frac{(1 + 4\sigma^2 + \sigma^4)(1 + \sigma^4)}{(1 + \sigma^2)^2}\label{cosineq10}
\end{align}

Using these, it is not difficult to see that $\eta_j(\phi) \leq \frac{1}{2}$ for all $\phi \in [0,\phi^*]$ and $j \in \N$ (i.e.~for all possible values of $\eta_j(\phi)$ in Table \ref{table2} and $\eta_j(\phi) \leq \frac{1}{4}$ for all $\phi \in [0,\phi^*]$ and $j \in \N$ with $t_j \notin \set{3,5}$. If
\[(1 + \sigma^2)^2 - 4\sigma^2\cos(\phi)^2 \leq (N - 2\sigma\cos(\phi))N,\]
then even $\eta_j(\phi) \leq \frac{1}{4}$ for all $\phi \in [0,\phi^*]$ and $j \in \N$ (i.e.~also if $t_j \in \set{3,5}$). In this case we can just choose $g_j = \frac{1}{2}$ for all $j \in \N$ and we are done. It thus remains to consider the case where
\[(1 + \sigma^2)^2 - 4\sigma^2\cos(\phi)^2 > (N - 2\sigma\cos(\phi))N.\]
The argument is now exactly the same as in the proof of Proposition \ref{case1}.

\begin{prop} \label{case2}
Let $\sigma \in (0,1]$, $U_{-1} = \left\{1\right\}$, $U_0 = \left\{0\right\}$, $U_1 = \left\{\pm \sigma\right\}$ and let $A \in \Psi E(U_{-1},U_0,U_1)$. Let $\phi \in [0,\phi^*]$, $\eta_j := \eta_j(\phi)$ and $t_j$ for all $j \in \N$ be defined as above. Further assume that
\[(1 + \sigma^2)^2 - 4\sigma^2\cos(\phi)^2 > (N - 2\sigma\cos(\phi))N.\]
Then the sequence $(g_j)_{j \in \N}$, defined  by the following prescription, satisfies $g_j \in [0,1]$ and $\eta_j \leq g_{j+1}(1 - g_j)$ for all $j \in \N$:
\begin{itemize}
	\item If $t_1 = 5$, choose $g_1 = 1 - \frac{1}{2}\frac{(1 + \sigma^2)^2 - 4\sigma^2\cos(\phi)^2}{(N - 2\sigma\cos(\phi))N}$.
	\item If there is some $k \in \N$ such that $t_1 = \ldots = t_k = 6$ and $t_{k+1} = 5$, choose $g_1 = 1 - \frac{1}{2}\frac{(1 + \sigma^2)^2 - 4\sigma^2\cos(\phi)^2}{(N - 2\sigma\cos(\phi))N}$.
	\item If neither is true, choose $g_1 = \frac{1}{2}$.
	\item If $t_j \in \set{2,6,10,14}$ and $t_{j+1} = 5$, choose $g_{j+1} = 1 - \frac{1}{2}\frac{(1 + \sigma^2)^2 - 4\sigma^2\cos(\phi)^2}{(N - 2\sigma\cos(\phi))N}$.
	\item If $t_j \in \set{2,6,10,14}$, there is some $k > j$ such that $t_{j+1} = \ldots = t_k = 6$ and $t_{k+1} = 5$, choose $g_{j+1} = 1 - \frac{1}{2}\frac{(1 + \sigma^2)^2 - 4\sigma^2\cos(\phi)^2}{(N - 2\sigma\cos(\phi))N}$.
	\item If $t_j = 3$, choose $g_{j+1} = \frac{1}{2}\frac{(1 + \sigma^2)^2 - 4\sigma^2\cos(\phi)^2}{(N - 2\sigma\cos(\phi))N}$.
	\item If $t_j = 11$, there is some $k \leq j$ such that $t_k = \ldots = t_j = 11$ and $t_{k-1} = 3$, choose $g_{j+1} = \frac{1}{2}\frac{(1 + \sigma^2)^2 - 4\sigma^2\cos(\phi)^2}{(N - 2\sigma\cos(\phi))N}$.
	\item If none of the above is true, choose $g_{j+1} = \frac{1}{2}$.
\end{itemize}
\end{prop}

\begin{proof}
The proof is exactly the same as the proof of Proposition \ref{case1}. We only have to change the numbers. That $g_j \in [0,1]$ holds for all $j \in \N$ follows from \eqref{cosineq5}, \eqref{cosineq7} and \eqref{cosineq10}. So it remains to prove $\eta_j \leq g_{j+1}(1 - g_j)$. Above we observed that $\eta_j \leq \frac{1}{4}$ unless $t_j \in \set{3,5}$. Thus if the cases $t_j = 3$ and $t_j = 5$ do not occur, then $\eta_j \leq g_{j+1}(1 - g_j)$ is obviously satisfied. So we are left with the cases $t_j = 3$ and $t_j = 5$ again.

Let us consider the case $t_j = 3$ first and assume that $g_j = \frac{1}{2}$. Then by definition
\[g_{j+1} = \frac{1}{2}\frac{(1 + \sigma^2)^2 - 4\sigma^2\cos(\phi)^2}{(N - 2\sigma\cos(\phi))N}\]
and
\[g_{j+1}(1 - g_j) = \frac{1}{4}\frac{(1 + \sigma^2)^2 - 4\sigma^2\cos(\phi)^2}{(N - 2\sigma\cos(\phi))N} = \eta_j.\]
Now there are four possible cases for $\eta_{j+1}$:
\begin{align*}
\eta_{j+1} &= \frac{(1 - \sigma^2)^2 + 4\sigma^2\cos(\phi)^2}{4(N - 2\sigma\cos(\phi))N} \quad (t_{j+1} = 9),\\
\eta_{j+1} &= \frac{(1 - \sigma^2)^2 + 4\sigma^2\cos(\phi)^2}{4N^2} \quad (t_{j+1} = 10),\\
\eta_{j+1} &= \frac{(1 + \sigma^2)^2 - 4\sigma^2\cos(\phi)^2}{4N^2} \quad (t_{j+1} = 11),\\
\eta_{j+1} &= \frac{(1 + \sigma^2)^2 - 4\sigma^2\cos(\phi)^2}{4(N+ 2\sigma\cos(\phi))N} (t_{j+1} = 12).
\end{align*}
In the first case we have $g_{j+2} = \frac{1}{2}$:
\begin{align*}
g_{j+2}(1 - g_{j+1}) &= \frac{1}{2} - \frac{1}{4}\frac{(1 + \sigma^2)^2 - 4\sigma^2\cos(\phi)^2}{(N - 2\sigma\cos(\phi))N}\\
&= \frac{1}{4}\frac{2(N - 2\sigma\cos(\phi))N - (1 + \sigma^2)^2 + 4\sigma^2\cos(\phi)^2}{(N - 2\sigma\cos(\phi))N}\\
&\geq \frac{1}{4}\frac{2(1 + \sigma^4) - (1 + \sigma^2)^2 + 4\sigma^2\cos(\phi)^2}{(N - 2\sigma\cos(\phi))N}\\
&= \frac{1}{4}\frac{(1 - \sigma^2)^2 + 4\sigma^2\cos(\phi)^2}{(N - 2\sigma\cos(\phi))N}\\
&= \eta_{j+1},
\end{align*}
where we used \eqref{cosineq5} and \eqref{cosineq7} in line 2. In the second case we have $g_{j+2} = 1 - \frac{1}{2}\frac{(1 + \sigma^2)^2 - 4\sigma^2\cos(\phi)^2}{(N - 2\sigma\cos(\phi))N} \leq \frac{1}{2}$ if $t_{j+2} \in \set{5,6}$ and $g_{j+2} = \frac{1}{2}$ if not:
\begin{align*}
g_{j+2}(1 - g_{j+1}) &\geq \left(1 - \frac{1}{2}\frac{(1 + \sigma^2)^2 - 4\sigma^2\cos(\phi)^2}{(N - 2\sigma\cos(\phi))N}\right)\left(1 - \frac{1}{2}\frac{(1 + \sigma^2)^2 - 4\sigma^2\cos(\phi)^2}{(N - 2\sigma\cos(\phi))N}\right)\\
&= \frac{1}{4}\left(\frac{(2N(N - 2\sigma\cos(\phi)) - (1 + \sigma^2)^2 + 4\sigma^2\cos(\phi)^2}{(N - 2\sigma\cos(\phi))N}\right)^2\\
&\geq \frac{1}{4}\left(\frac{(N(N - 2\sigma\cos(\phi)) + 1 + \sigma^4 - (1 + \sigma^2)^2 + 4\sigma^2\cos(\phi)^2}{(N - 2\sigma\cos(\phi))N}\right)^2\\
&\geq \frac{1}{4}\left(\frac{(N(N - 2\sigma\cos(\phi)) - 2\sigma N\cos(\phi) + 4\sigma^2\cos(\phi)^2}{(N - 2\sigma\cos(\phi))N}\right)^2\\
&= \frac{1}{4} \frac{(N - 2\sigma\cos(\phi))^2}{N^2}\\
&= \frac{1}{4} \frac{(\sqrt{(1-\sigma^2)^2 + 4\sigma^2\cos(\phi)^2})^2}{N^2}\\
&= \eta_{j+1},
\end{align*}
where we used \eqref{cosineq5} and \eqref{cosineq7} in line 2 and \eqref{cosineq5} and \eqref{cosineq6} in line 3. In the third case we have $g_{j+2} = \frac{1}{2}\frac{(1 + \sigma^2)^2 - 4\sigma^2\cos(\phi)^2}{(N - 2\sigma\cos(\phi))N}$:
\begin{align*}
g_{j+2}(1 - g_{j+1}) &= \frac{1}{2}\frac{(1 + \sigma^2)^2 - 4\sigma^2\cos(\phi)^2}{(N - 2\sigma\cos(\phi))N}\left(1 - \frac{1}{2}\frac{(1 + \sigma^2)^2 - 4\sigma^2\cos(\phi)^2}{(N - 2\sigma\cos(\phi))N}\right)\\
&\geq \frac{1}{4}\frac{(1 + \sigma^2)^2 - 4\sigma^2\cos(\phi)^2}{(N - 2\sigma\cos(\phi))N}\frac{N - 2\sigma\cos(\phi)}{N}\\
&= \frac{1}{4}\frac{(1 + \sigma^2)^2 - 4\sigma^2\cos(\phi)^2}{N^2}\\
&= \eta_{j+1}
\end{align*}
like in the second case. In the fourth case we have $g_{j+2} = \frac{1}{2}$:
\begin{align*}
g_{j+2}(1 - g_{j+1}) &= \frac{1}{2} - \frac{1}{4}\frac{(1 + \sigma^2)^2 - 4\sigma^2\cos(\phi)^2}{(N - 2\sigma\cos(\phi))N}\\
&\geq \frac{1}{2} - \frac{1}{4}\frac{1 + 4\sigma^2 + \sigma^4}{(1 + \sigma^2)^2}\\
&= \frac{1}{4}\frac{1 + \sigma^4}{(1 + \sigma^2)^2}\\
&\geq \frac{1}{4}\frac{(1 + \sigma^2)^2 - 4\sigma^2\cos(\phi)^2}{(N+ 2\sigma\cos(\phi))N}\\
&= \eta_{j+1},
\end{align*}
where we used \eqref{cosineq5}, \eqref{cosineq7} and \eqref{cosineq10} in line 1 and \eqref{cosineq5}, \eqref{cosineq8} and \eqref{cosineq10} in line 3. So either $g_{j+2} \leq \frac{1}{2}$ or $g_{j+2} = g_{j+1}$. As in the proof of Proposition \ref{case1} we conclude that we eventually go out with $g_k \leq \frac{1}{2}$ for some $k \geq j+2$. Thus we are done by induction if we can control the case $t_j = 5$ as well.

So assume $t_j = 5$. Then $g_j = 1 - \frac{1}{2}\frac{(1 + \sigma^2)^2 - 4\sigma^2\cos(\phi)^2}{(N - 2\sigma\cos(\phi))N}$ and $g_{j+1} = \frac{1}{2}$ by definition and thus
\[g_{j+1}(1 - g_j) = \frac{1}{4}\frac{(1 + \sigma^2)^2 - 4\sigma^2\cos(\phi)^2}{(N - 2\sigma\cos(\phi))N} = \eta_j.\]
Again there are five cases here. The first case is $j = 1$, which is again trivial. The second case is where $t_{j-1} = 2$. In this case we have $g_{j-1} = \frac{1}{2}$:
\begin{align*}
g_j(1 - g_{j-1}) &= \frac{1}{2} - \frac{1}{4}\frac{(1 + \sigma^2)^2 - 4\sigma^2\cos(\phi)^2}{(N - 2\sigma\cos(\phi))N}\\
&= \frac{1}{4}\frac{2(N - 2\sigma\cos(\phi))N - (1 + \sigma^2)^2 + 4\sigma^2\cos(\phi)^2}{(N - 2\sigma\cos(\phi))N}\\
&\geq \frac{1}{4}\frac{2(1 + \sigma^4) - (1 + \sigma^2)^2 + 4\sigma^2\cos(\phi)^2}{(N - 2\sigma\cos(\phi))N}\\
&= \frac{1}{4}\frac{(1 - \sigma^2)^2 + 4\sigma^2\cos(\phi)^2}{(N - 2\sigma\cos(\phi))N}\\
&= \eta_{j+1},
\end{align*}
where we used \eqref{cosineq5} and \eqref{cosineq7} in line 2. The third case is where $t_{j-1} = 6$. In this case we have $g_{j-1} = 1 - \frac{1}{2}\frac{(1 + \sigma^2)^2 - 4\sigma^2\cos(\phi)^2}{(N - 2\sigma\cos(\phi))N}$:
\begin{align*}
g_j(1 - g_{j-1}) &= \left(1 - \frac{1}{2}\frac{(1 + \sigma^2)^2 - 4\sigma^2\cos(\phi)^2}{(N - 2\sigma\cos(\phi))N}\right)\frac{1}{2}\frac{(1 + \sigma^2)^2 - 4\sigma^2\cos(\phi)^2}{(N - 2\sigma\cos(\phi))N}\\
&= \frac{1}{4}\frac{(2N(N - 2\sigma\cos(\phi)) - (1 + \sigma^2)^2 + 4\sigma^2\cos(\phi)^2}{(N - 2\sigma\cos(\phi))N}\frac{(1 + \sigma^2)^2 - 4\sigma^2\cos(\phi)^2}{(N - 2\sigma\cos(\phi))N}\\
&\geq \frac{1}{4}\frac{(N(N - 2\sigma\cos(\phi)) + 1 + \sigma^4 - (1 + \sigma^2)^2 + 4\sigma^2\cos(\phi)^2}{(N - 2\sigma\cos(\phi))N}\frac{(1 + \sigma^2)^2 - 4\sigma^2\cos(\phi)^2}{(N - 2\sigma\cos(\phi))N}\\
&\geq \frac{1}{4}\frac{(N(N - 2\sigma\cos(\phi)) - 2\sigma N\cos(\phi) + 4\sigma^2\cos(\phi)^2}{(N - 2\sigma\cos(\phi))N}\frac{(1 + \sigma^2)^2 - 4\sigma^2\cos(\phi)^2}{(N - 2\sigma\cos(\phi))N}\\
&= \frac{1}{4} \frac{N - 2\sigma\cos(\phi)}{N}\frac{(1 + \sigma^2)^2 - 4\sigma^2\cos(\phi)^2}{(N - 2\sigma\cos(\phi))N}\\
&= \frac{1}{4}\frac{(1 + \sigma^2)^2 - 4\sigma^2\cos(\phi)^2}{N^2}\\
&= \eta_{j-1},
\end{align*}
where we used \eqref{cosineq5} and \eqref{cosineq7} in line 2 and \eqref{cosineq5} and \eqref{cosineq6} in line 3. The fourth case is where $t_{j-1} = 10$. In this case we either have $g_{j-1} = \frac{1}{2}\frac{(1 + \sigma^2)^2 - 4\sigma^2\cos(\phi)^2}{(N - 2\sigma\cos(\phi))N} \geq \frac{1}{2}$ if $t_{j-2} \in \set{3,11}$ or $g_{j-1} = \frac{1}{2}$ if not:
\begin{align*}
g_j(1 - g_{j-1}) &\geq \left(1 - \frac{1}{2}\frac{(1 + \sigma^2)^2 - 4\sigma^2\cos(\phi)^2}{(N - 2\sigma\cos(\phi))N}\right)\left(1 - \frac{1}{2}\frac{(1 + \sigma^2)^2 - 4\sigma^2\cos(\phi)^2}{(N - 2\sigma\cos(\phi))N}\right)\\
&= \frac{1}{4} \frac{(N - 2\sigma\cos(\phi))^2}{N^2}\\
&= \frac{1}{4} \frac{(\sqrt{(1-\sigma^2)^2 + 4\sigma^2\cos(\phi)^2})^2}{N^2}\\
&= \eta_{j+1}.
\end{align*}
The fifth case is where $t_{j-1} = 14$. In this case we have $g_{j-1} = \frac{1}{2}$:
\begin{align*}
g_j(1 - g_{j-1}) &= \frac{1}{2} - \frac{1}{4}\frac{(1 + \sigma^2)^2 - 4\sigma^2\cos(\phi)^2}{(N - 2\sigma\cos(\phi))N}\\
&\geq \frac{1}{2} - \frac{1}{4}\frac{1 + 4\sigma^2 + \sigma^4}{(1 + \sigma^2)^2}\\
&= \frac{1}{4}\frac{1 + \sigma^4}{(1 + \sigma^2)^2}\\
&\geq \frac{1}{4}\frac{(1 + \sigma^2)^2 - 4\sigma^2\cos(\phi)^2}{(N+ 2\sigma\cos(\phi))N}\\
&= \eta_{j+1},
\end{align*}
where we used \eqref{cosineq5}, \eqref{cosineq7} and \eqref{cosineq10} in line 1 and \eqref{cosineq5}, \eqref{cosineq8} and \eqref{cosineq10} in line 3. As in the proof of Proposition \ref{case1} we conclude that we either end up at $g_1$, which is fine or we eventually have $g_k \geq \frac{1}{2}$ for some $k \leq j-1$. In either case we are done by induction.
\end{proof}

Using the sequences obtained in Proposition \ref{case1} and Proposition \ref{case2}, we can now apply Lemma \ref{SzwarcLemma1} to prove Theorem \ref{N(A^2)}:

\begin{proof}[Proof of Theorem \ref{N(A^2)}]
Let $A \in \PsiE(U_{-1},U_0,U_1)$. The inclusion
\[N(A^2) \supseteq \conv\left(N(B_1^2) \cup N(B_2^2) \cup N(B_3^2)\right)\]
is clear by Theorem \ref{mtnr} and the fact that $\sigma^{\op}(B^2) = \sigma^{\op}(B)^2$ (see Proposition \ref{basic}). To prove the other inclusion, we have to show $r_\phi(A^2) \leq N(\phi)$ for all $\phi \in [0,2\pi)$, where $N(\phi)$ is given by Proposition \ref{ValuesOfN}. Using the transformations $\phi \mapsto \pi - \phi$ and $\phi \mapsto \phi + \pi$, it is clear that is suffices to consider $\phi \in [0,\frac{\pi}{2}]$. Indeed, $N(\phi)$ is invariant under these transformations and in the Tables \ref{table1} and \ref{table2} only the roles of $+\sigma$ and $-\sigma$ are interchanged. To apply Lemma \ref{SzwarcLemma1} to $E(\phi)$, we have to assure
\[E_{j,j+1}(\phi) = \frac{1}{2}\abs{e^{i\phi}(C_+)_{j,j+1} + e^{-i\phi}\overline{(C_+)_{j,j+1}}} > 0\]
and $E_{j,j}(\phi) > 0$ for all $\phi \in [0,\frac{\pi}{2}]$. The latter can be achieved by shifting and the former can only fail if $\sigma = 1$ and $\phi = 0$. But in this case we trivially have
\[r_0(E(\phi)) \leq \norm{E(\phi)} \leq 4 = N(\phi)\]
by the Wiener estimate (e.g.~\cite[p.~25]{Marko}). Moreover, we clearly have $N(\phi) > \sup\limits_{j \in \N} E_{j,j}(\phi)$ as $E_{j,j} \in \set{-2\sigma\cos(\phi),0,2\sigma\cos(\phi)}$ for all $j \in \N$ and $\phi \in [0,\frac{\pi}{2}]$ (cf.~Proposition \ref{ValuesOfN}). We can thus apply Lemma \ref{SzwarcLemma1}, using the sequences from Proposition \ref{case1} and Proposition \ref{case2} (including the trivial case where $(1 + \sigma^2)^2 - 4\sigma^2\cos(\phi)^2 \leq (N - 2\sigma\cos(\phi))N$), to obtain $r_\phi(A^2) = r_0(E(\phi)) \leq N(\phi)$ for all $\phi \in [0,\frac{\pi}{2}]$ and hence all $\phi \in [0,2\pi)$.

The inclusion for more general operators $A \in M(U_{-1},U_0,U_1)$ now follows from Theorem \ref{mtnr} and Proposition \ref{basic} again.
\end{proof}

In Figure \ref{sigma=1} we can see that $\sqrt{N(A^2)}$ is indeed a tighter upper bound to the spectrum than $N(A)$. Moreover, it shows that $\spec(A)$ is not equal to $N(A)$ and thus not convex. This confirmes and improves the numerical results obtained in \cite{ChaChoLi2}. A rigorous proof of this observation can be found in Section \ref{ExplForm}.

\begin{figure}[ht]
\begin{center}
\includegraphics[scale = 1]{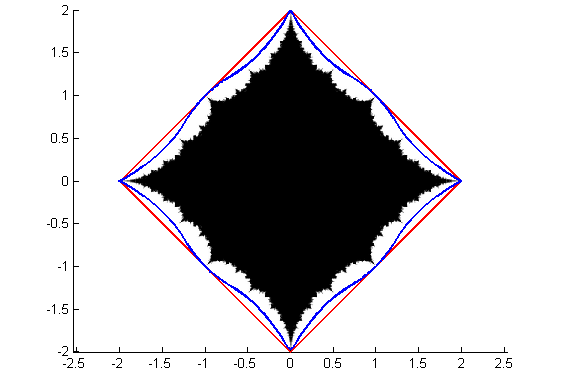}
\end{center}

\caption{The boundary of $\sqrt{N(A^2)}$ (blue), the boundary of $N(A)$ (red) and a lower bound to $\spec(A)$ consisting of spectra of periodic operators and the closed unit disk (black, see \cite{ChaChoLi},\cite{ChaChoLi2}) in the case $\sigma = 1$.} \label{sigma=1}
\end{figure}

\subsection{\texorpdfstring{A proof that $\sqrt{N(A^2)} \subset N(A)$}{}} \label{ExplForm}

In this section we provide formulas for $N(A)$, $N(A)^2$ and $N(A^2)$ in terms of graphs of explicit functions. These follow from elementary computations using Theorem \ref{ntridiag}, Theorem \ref{N(A^2)} and Proposition \ref{B1B2B3}. These formulas then allow us to show that $\sqrt{N(A^2)}$ is indeed a proper subset of $N(A)$.

\begin{prop} \label{nrFZ}
Let $\sigma \in (0,1]$, $U_{-1} = \left\{1\right\}$, $U_0 = \left\{0\right\}$, $U_1 = \left\{\pm \sigma\right\}$ and $A \in \PsiE(U_{-1},U_0,U_1)$. Then
\[N(A) = \set{x + iy \in \C : -f(x) \leq y \leq f(x), -(1+\sigma) \leq x \leq 1+\sigma},\]
where $f \from [-(1+\sigma),1+\sigma] \to \R$ is given by
\[f(x) = \begin{cases} (1-\sigma)\sqrt{1-(\frac{x}{1+\sigma})^2} & \text{for } x \in \left[-(1+\sigma),-\frac{(1+\sigma)^2}{\sqrt{2(1+\sigma^2)}}\right],\\
\sqrt{2(1+\sigma^2)} + x & \text{for } x \in \left(-\frac{(1+\sigma)^2}{\sqrt{2(1+\sigma^2)}},-\frac{(1-\sigma)^2}{\sqrt{2(1+\sigma^2)}}\right],\\
(1+\sigma)\sqrt{1-(\frac{x}{1-\sigma})^2} & \text{for } x \in \left(-\frac{(1-\sigma)^2}{\sqrt{2(1+\sigma^2)}},\frac{(1-\sigma)^2}{\sqrt{2(1+\sigma^2)}}\right),\\
\sqrt{2(1+\sigma^2)} - x & \text{for } x \in \left[\frac{(1-\sigma)^2}{\sqrt{2(1+\sigma^2)}},\frac{(1+\sigma)^2}{\sqrt{2(1+\sigma^2)}}\right),\\
(1-\sigma)\sqrt{1-(\frac{x}{1+\sigma})^2} & \text{for } x \in \left[\frac{(1+\sigma)^2}{\sqrt{2(1+\sigma^2)}},1+\sigma\right]. \end{cases}\]
\end{prop}

\begin{proof}
By Theorem \ref{ntridiag}, the numerical range of $A$ is given by the convex hull of the two ellipses $\set{e^{i\theta} + \sigma e^{-i\theta} : \theta \in [0,2\pi)}$ and $\set{e^{i\theta} - \sigma e^{-i\theta} : \theta \in [0,2\pi)}$. The assertion thus follows by an elementary computation.
\end{proof}

\begin{prop} \label{(nrFZ)^2}
Let $\sigma \in (0,1]$, $U_{-1} = \left\{1\right\}$, $U_0 = \left\{0\right\}$, $U_1 = \left\{\pm \sigma\right\}$ and $A \in \PsiE(U_{-1},U_0,U_1)$. Then 
\[N(A)^2 = \set{x + iy \in \C : -f(x) \leq y \leq f(x), -(1+\sigma)^2 \leq x \leq (1+\sigma)^2},\]
where $f \from [-(1+\sigma)^2,(1+\sigma)^2] \to \R$ is given by
\[f(x) = \begin{cases} (1-\sigma^2)\sqrt{1-(\frac{x+2\sigma}{1+\sigma^2})^2} & \text{for } x \in \left[-(1+\sigma)^2,-4\sigma\right),\\
1 + \sigma^2 - \frac{x^2}{4(1+\sigma^2)} & \text{for } x \in \left[-4\sigma,4\sigma\right],\\
(1-\sigma^2)\sqrt{1-(\frac{x-2\sigma}{1+\sigma^2})^2} & \text{for } x \in \left(4\sigma,(1+\sigma)^2\right].\end{cases}\]
\end{prop}

\begin{proof}
Using $\Re(z^2) = (\Re z)^2 - (\Im z)^2$ and $\Im(z^2) = 2\Re z\Im z$ for $z \in \C$, this follows from Proposition \ref{nrFZ} by another elementary computation.
\end{proof}

\begin{prop} \label{nrFZ^2}
Let $\sigma \in (0,1]$, $U_{-1} = \left\{1\right\}$, $U_0 = \left\{0\right\}$, $U_1 = \left\{\pm \sigma\right\}$ and $A \in \PsiE(U_{-1},U_0,U_1)$. Then
\[N(A^2) = \set{x + iy \in \C : -g(x) \leq y \leq g(x), -(1+\sigma)^2 \leq x \leq (1+\sigma)^2},\] where $g \from [-(1+\sigma)^2,(1+\sigma)^2] \to \R$ is given by
\[g(x) = \begin{cases} (1-\sigma^2)\sqrt{1-(\frac{x+2\sigma}{1+\sigma^2})^2} & \text{for } x \in \left[-(1+\sigma)^2,-2\sigma - \sigma\frac{(1+\sigma^2)^2}{1+\sigma^4}\right),\\
\frac{(1+\sigma^2)^2}{\sqrt{1+\sigma^2+\sigma^4}} + \frac{\sigma}{\sqrt{1+\sigma^2+\sigma^4}}x & \text{for } x \in \left[-2\sigma - \sigma\frac{(1+\sigma^2)^2}{1+\sigma^4},-\sigma\right),\\
\sqrt{(1+\sigma^2)^2 - x^2} & \text{for } x \in \left[-\sigma,\sigma\right],\\
\frac{(1+\sigma^2)^2}{\sqrt{1+\sigma^2+\sigma^4}} - \frac{\sigma}{\sqrt{1+\sigma^2+\sigma^4}}x & \text{for } x \in \left[\sigma,2\sigma + \sigma\frac{(1+\sigma^2)^2}{1+\sigma^4}\right),\\
(1-\sigma^2)\sqrt{1-(\frac{x-2\sigma}{1+\sigma^2})^2} & \text{for } x \in \left[2\sigma + \sigma\frac{(1+\sigma^2)^2}{1+\sigma^4},(1+\sigma)^2\right).\end{cases}\]
\end{prop}

\begin{proof}
This follows from Theorem \ref{N(A^2)} and Proposition \ref{B1B2B3} by yet another tedious but elementary computation.
\end{proof}

Thus $N(A)^2$ is surrounded by (parts of) two parabolas and two ellipses whereas $N(A^2)$ is surrounded by (parts of) a circle, two ellipses and four straight lines (see Figure \ref{sigma=0.5} for the case $\sigma = \frac{1}{2}$). It is readily seen that the ellipses are the same, respectively.

\begin{figure}
\begin{center}
\includegraphics[scale = 0.7]{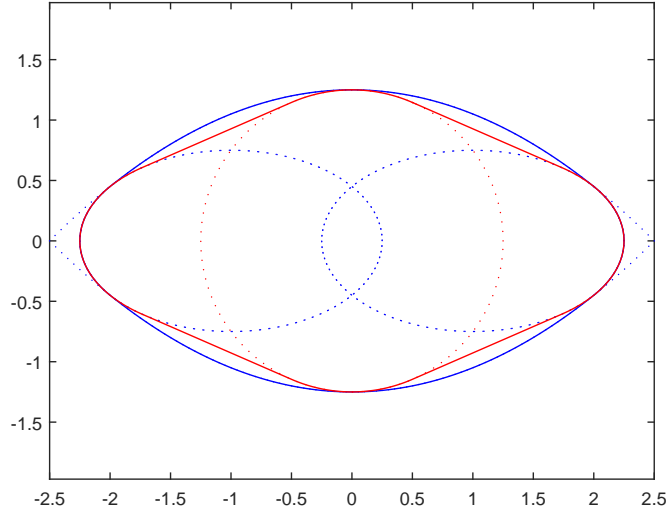}
\end{center}

\caption{The two parabolas and the two ellipses (blue, dotted), the circle (red, dotted), $N(A)^2$ (blue, solid) and $N(A^2)$ (red, solid) in the case $\sigma = \frac{1}{2}$.} \label{sigma=0.5}
\end{figure}

\begin{thm} \label{ProperSubset}
Let $\sigma \in (0,1]$, $U_{-1} = \left\{1\right\}$, $U_0 = \left\{0\right\}$, $U_1 = \left\{\pm \sigma\right\}$ and $A \in \PsiE(U_{-1},U_0,U_1)$. Then $N(A^2)$ is a proper subset of $N(A)^2$.
\end{thm}

\begin{proof}
Let $f$ be as in Proposition \ref{(nrFZ)^2} and $B_1$, $B_2$, $B_3$ as in Theorem \ref{N(A^2)}. We will show that $f$ is concave, which implies that $N(A)^2$ is convex. It then remains to show that $N(A)^2$ contains $N(B_1^2)$, $N(B_2^2)$ and $N(B_3^2)$ by Theorem \ref{N(A^2)}. Using Corollary \ref{mtnrpe} and the parametrizations of $\partial N(B_1^2)$ and $\partial N(B_3^2)$ provided by Proposition \ref{B1B2B3}, it is easily seen that $N(B_1^2) = N(B_1)^2 \subset N(A)^2$ and $N(B_3^2) = N(B_3)^2 \subset N(A)^2$. It will thus suffice to consider $N(B_2^2)$.

Clearly, $f$ is continuously differentiable with
\[f'(x) = \begin{cases} 2\frac{x+2\sigma}{1+\sigma^2}\frac{1-\sigma^2}{1+\sigma^2}\left(1 - (\frac{x+2\sigma}{1+\sigma^2})^2\right)^{-1/2} & \text{for } x \in \left[-(1+\sigma)^2,-4\sigma\right),\\
-\frac{x}{2(1+\sigma^2)} & \text{for } x \in \left[-4\sigma,4\sigma\right],\\
2\frac{x-2\sigma}{1+\sigma^2}\frac{1-\sigma^2}{1+\sigma^2}\left(1 - (\frac{x-2\sigma}{1+\sigma^2})^2\right)^{-1/2} & \text{for } x \in \left(4\sigma,(1+\sigma)^2\right].\end{cases}\]
Moreover, $f'$ is piecewise continuously differentiable with
\[f''(x) = \begin{cases} -\frac{1 - \sigma^2}{1 + \sigma^2}\left(1 - (\frac{x+2\sigma}{1+\sigma^2})^2\right)^{-3/2} & \text{for } x \in \left[-(1+\sigma)^2,-4\sigma\right),\\
-\frac{1}{2(1+\sigma^2)} & \text{for } x \in \left(-4\sigma,4\sigma\right),\\
-\frac{1 - \sigma^2}{1 + \sigma^2}\left(1 - (\frac{x-2\sigma}{1+\sigma^2})^2\right)^{-3/2} & \text{for } x \in \left(4\sigma,(1+\sigma)^2\right].\end{cases}\]
Thus $f''(x) < 0$ for $x \in [-(1+\sigma)^2,(1+\sigma)^2] \setminus \set{-4\sigma,4\sigma}$, which implies that $f$ is concave.

Let $g \from [-(1+\sigma^2),1+\sigma^2] \to \R$ be defined by $\sqrt{(1+\sigma^2)^2 - x^2}$ so that
\[N(B_2^2) = \set{x + iy \in \C : -g(x) \leq y \leq g(x), -(1+\sigma^2) \leq x \leq 1+\sigma^2}\]
(see Proposition \ref{B1B2B3}). Assume first that $4\sigma \geq 1+\sigma^2$. Then
\begin{align*}
f(x) = g(x) &\Leftrightarrow 1 + \sigma^2 - \frac{x^2}{4(1+\sigma^2)} = \sqrt{(1+\sigma^2)^2 - x^2}\\
&\Leftrightarrow \left(1 + \sigma^2 - \frac{x^2}{4(1+\sigma^2)}\right)^2 = (1+\sigma^2)^2 - x^2\\
&\Leftrightarrow \frac{x^2}{2} + \frac{x^4}{16(1+\sigma^2)^2} = 0\\
&\Leftrightarrow x = 0
\end{align*}
for $x \in [-(1+\sigma^2),1+\sigma^2]$. Thus the graphs of $f$ and $g$ only intersect at $x = 0$. Since both $f$ and $g$ are continuous, it suffices to plug in some values (e.g.~$\pm (1+\sigma^2)$) to conclude $f \geq g$ and thus $N(B_2^2) \subseteq N(A)^2$. As we mentioned at the beginning of the proof, this implies $N(A^2) \subseteq N(A)^2$.

Now let $4\sigma < 1+\sigma^2$. For $x \in [-4\sigma,4\sigma]$, this is the same as above. For $x \in (4\sigma,1+\sigma^2]$ we have
\begin{align*}
f(x) = g(x) &\Leftrightarrow (1 - \sigma^2)\sqrt{1 - \left(\frac{x - 2\sigma}{1 + \sigma^2}\right)^2} = \sqrt{(1+\sigma^2)^2 - x^2}\\
&\Leftrightarrow (1 - \sigma^2)^2\left(1 - \left(\frac{x - 2\sigma}{1 + \sigma^2}\right)^2\right) = (1+\sigma^2)^2 - x^2.
\end{align*}
But this quadratic equation only has the solutions $x = 2\sigma$ and $x = -\frac{1 + \sigma^4}{\sigma}$, which are not contained in $(4\sigma,1+\sigma^2]$. Thus the graphs of $f$ and $g$ do not intersect in $(4\sigma,1+\sigma^2]$. Similarly, the graphs of $f$ and $g$ do not intersect in $[-(1+\sigma^2),-4\sigma)$. Since $f$ and $g$ are continuous, this again implies that $f \geq g$ and thus $N(A^2) \subseteq N(A)^2$.

It is now easily seen that this inclusion has to be proper.
\end{proof}

Since $N(A)$ is symmetric w.r.t.~the origin (cf. Proposition \ref{nrFZ}), Theorem \ref{ProperSubset} implies that $\sqrt{N(A^2)}$ is indeed a tighter upper bound to $\spec(A)$ than $N(A)$.

\begin{cor}
Let $\sigma \in (0,1]$, $U_{-1} = \left\{1\right\}$, $U_0 = \left\{0\right\}$, $U_1 = \left\{\pm \sigma\right\}$ and $A \in \PsiE(U_{-1},U_0,U_1)$. Then $\sqrt{N(A^2)}$ is a proper subset of $N(A)$.
\end{cor}

\bigskip

\textbf{Acknowledgments.}
The author wants to thank the referee for his very detailed and accurate review and many useful comments. Moreover, the author appreciates the continuous support by Marko Lindner (TUHH).

\bigskip

\noindent
{\bf Author's address:}

\medskip

\noindent
Raffael Hagger \hfill{\tt raffael.hagger@tuhh.de}\\
Institute of Mathematics\\
Hamburg University of Technology\\
Schwarzenbergstr. 95 E\\
D-21073 Hamburg\\
GERMANY


\newpage

\begin{thebibliography}{}
  \bibitem{Anderson} P.W.~Anderson: \emph{Absence of Diffusion in certain Random Lattices}, Phys.~Rev.~109, 1492-1505 (1958).
  \bibitem{ChaChoLi} S.N.~Chandler-Wilde, R.~Chonchaiya and M.~Lindner: \emph{Eigenvalue Problem meets Sierpinski Triangle: Computing the Spectrum of a Non-Self-Adjoint Random Operator}, Operators and Matrices, 5 (2011), 633-648.
  \bibitem{ChaChoLi2} S.N.~Chandler-Wilde, R.~Chonchaiya and M.~Lindner: \emph{On the Spectra and Pseudospectra of a Class of non-self-adjoint Random Matrices and Operators}, Operators and Matrices, 7 (2013), 739-775.
  \bibitem{ChaDa} S.N.~Chandler-Wilde and E.B.~Davies: \emph{Spectrum of a Feinberg-Zee Random Hopping Matrix}, Journal of Spectral Theory, 2 (2012), 147-179.
  \bibitem{ChaLi} S.N.~Chandler-Wilde and M.~Lindner: \emph{Limit Operators, Collective Compactness, and the Spectral Theory of Infinite Matrices}, Mem. Amer. Math. Soc. 210 (2011), No. 989.
  \bibitem{Davies} E.B.~Davies: \emph{Spectral Theory of Pseudo-Ergodic Operators}, Commun.~Math.~Phys.~216 (2001), 687-704.
  \bibitem{Davies0.5} E.B.~Davies: \emph{Spectral Bounds Using Higher Order Numerical Ranges}, LMS J.~Comput.~Math.~8 (2005), 17-45.
	\bibitem{Davies2} E.B.~Davies: \emph{Linear Operators and their Spectra}, Cambridge Studies in Advanced Mathematics 106, Cambridge University Press, Cambridge, 2007.
	\bibitem{FeZee} J.~Feinberg and A.~Zee: \emph{Spectral Curves of Non-Hermitean Hamiltonians}, Nucl.~Phys.~B 552 (1999), 599-623.
	\bibitem{HaRoSi} R.~Hagen, S.~Roch and B.~Silbermann: \emph{$C^*$-Algebras and Numerical Analysis}, Marcel Dekker, Inc., New York, Basel 2001.
	\bibitem{Ha} R.~Hagger: \emph{The Eigenvalues of Tridiagonal Sign Matrices are Dense in the Spectra of Periodic Tridiagonal Sign Operators}, to appear in J.~Funct.~Anal., doi:10.1016/j.jfa.2015.01.019.
	\bibitem{Ha2} R.~Hagger: \emph{Symmetries of the Feinberg-Zee Random Hopping Matrix}, submitted, preprint at arXiv: 1412.1937.
	\bibitem{HaLiSe} R.~Hagger, M.~Lindner and M.~Seidel: \emph{Essential Pseudospectra and Essential Norms of Band-Dominated Operators}, submitted, preprint at arXiv: 1504.00540.
	\bibitem{HaNe} N.~Hatano, D.R.~Nelson: \emph{Localization Transitions in Non-Hermitian Quantum Mechanics}, Phys.~Rev.~Lett., Vol.~77, No.~3 (1996), 570-573.
	\bibitem{HoOrZee} D.E.~Holz, H.~Orland and A.~Zee: \emph{On the Remarkable Spectrum of a Non-Hermitian Random Matrix Model}, J.~Phys.~A Math.~Gen 36 (2003), 3385-3400.
	\bibitem{Johnson} C.R.~Johnson: \emph{Numerical Determination of the Field of Values of a General Complex Matrix}, SIAM J.~Numer.~Anal., Vol.~15, No.~3 (1978), 595-602.
	\bibitem{Marko} M.~Lindner: \emph{Infinite Matrices and their Finite Sections}, Birkhäuser Verlag, Basel, Boston, Berlin, 2006.
	\bibitem{HabilMarko} M.~Lindner: \emph{Fredholm Theory and Stable Approximation of Band Operators and Their Generalisations}, Habilitationsschrift, Chemnitz, 2009.
	\bibitem{MarkoBidiag} M.~Lindner: \emph{A Note on the Spectrum of Bi-infinite Bi-diagonal Random Matrices}, Journal of Analysis and Applications 7 (2009), 269-278.
	\bibitem{LiRo} M.~Lindner and S.~Roch: \emph{Finite Sections of Random Jacobi Operators}, SIAM J.~Numerical Analysis 50(1), 287-306 (2012).
	\bibitem{Martinez} C.~Martínez-Adame: \emph{Spectral Estimates for the One-Dimensional Non-Self-Adjoint Anderson Model}, J.~Operator Theory 56 (2006), 59-88.
	\bibitem{RaRoSi} V.S.~Rabinovich, S.~Roch and B.~Silbermann: \emph{Limit Operators and their Applications in Operator Theory}, Operator Theory: Advances and Applications, 150, Birkhäuser Verlag, Basel, 2004.
	\bibitem{Szwarc} R.~Szwarc: \emph{Norm Estimates of Discrete Schrödinger operators}, Colloq. Math. 76 (1998), 153-160.
	\bibitem{Wigner} E.P.~Wigner: \emph{Characteristic Vectors of Bordered Matrices with Infinite Dimensions}, Ann. of Math., 2nd Series, Vol. 62, No. 3 (1955), 548-564.
\end{thebibliography}
\end{document}